\documentclass{amsart}

\newtheorem{theorem}{Theorem}[section]
\newtheorem{lemma}[theorem]{Lemma}
\newtheorem{proposition}[theorem]{Proposition}

\newtheorem{assumption}[theorem]{Assumption}
\theoremstyle{definition}

\theoremstyle{remark}
\newtheorem{remark}[theorem]{Remark}

\numberwithin{equation}{section}

\usepackage{graphicx}
\graphicspath{{./image/}}
\usepackage{color}
\usepackage{fancyhdr}
\begin{document}
\title{Ellipticity of Bartnik boundary data for stationary vacuum spacetimes}
\author{Zhongshan An}
\begin{abstract}
We establish a moduli space $\mathbb E$ of stationary vacuum metrics in a spacetime, and set up a well-defined boundary map $\Pi$ in $\mathbb E$, assigning a metric class with its Bartnik boundary data. Furthermore, we prove the boundary map $\Pi$ is Fredholm by showing that the stationary vacuum equations (combined with proper gauge terms) and the Bartnik boundary conditions form an elliptic boundary value problem. As an application, we show that the Bartnik boundary data near the standard flat boundary data admits a unique (up to diffeomorphism) stationary vacuum extension locally.
\end{abstract}
\maketitle

\section{Introduction}

In general relativity, one of the most interesting and well-known notions of quasi-local mass is the Bartnik quasi-local mass. Let $\Omega$ be a bounded smooth 3-manifold with nonempty boundary $\Sigma$. Equip $\Omega$ with a Riemannian metric $g$ and a symmetric 2-tensor $K$, which is essentially the second fundamental form of $\Omega$ when it is embedded in some spacetime. The Bartnik quasi-local mass of the data set $(\Omega,g,K)$ is defined as (cf.[B1],[B2]),
\begin{equation}
m_B[(\Omega,g,K)]=\text{inf}\{m_{ADM}[(M,g,K)]\},
\end{equation}
where the infimum is taken over all asymptotically flat admissible initial data sets $(M,g,K)$ such that after gluing $M$ and $\Omega$ along the boundary $\partial M\cong\partial \Omega$, the manifold $M\cup\Omega$ forms a complete asymptotically flat spacetime initial data set. 

By analyzing the constraint equations across the boundary $\Sigma=\partial\Omega\cong\partial M$, Bartnik proposed a set of geometric boundary data for $(\Omega,g,K)$ given by,
\begin{equation}
(g_{\Sigma}, H_{\Sigma}, tr_{\Sigma}K,\omega_{\mathbf n_{\Sigma}}).
\end{equation}
Here $g_{\Sigma}$ is the induced metric on the boundary $\Sigma$ obtained from $(\Omega, g)$; $H_{\Sigma}$ is the mean curvature of $\Sigma\subset (\Omega,g)$; $tr_{\Sigma}K$ is the trace of the restriction $K|_{\Sigma}$ of the second fundamental form; and $\omega_{\mathbf n_{\Sigma}}$ is the connection 1-form of the spacetime normal bundle of $\Sigma$, which is defined as,
$$\omega_{\mathbf n_{\Sigma}}(v)=K(\mathbf n_{\Sigma},v),~\forall v\in T\Sigma,$$
where $\mathbf n_{\Sigma}$ is the outward unit normal vector field on $\Sigma\subset (\Omega,g)$.

Then definition (1.1) of the Bartnik quasi-local mass can be reduced to the infimum ADM mass taken over all asymptotically flat admissible initial data sets $(M,g,K)$ which satisfy the following boundary conditions:
\begin{equation}
\begin{cases}
g_{\partial M}=g_{\Sigma}\\
H_{\partial M}=H_{\Sigma}\\
tr_{\partial M}K=tr_{\Sigma}K\\
\omega_{\mathbf n_{\Sigma}}=\omega_{\mathbf n_{\partial M}}.
\end{cases}
\end{equation}

The tuple of geometric boundary data (1.2) is called the Bartnik boundary data. It also arises naturally from a Hamiltonian analysis of the vacuum Einstein equations. In fact, a regularization $\mathcal H$ of the Regge-Teitelboim Hamiltonian is constructed in [B3]. When the spacetime has empty boundary, by analyzing the functional $\mathcal H$ and following an approach initiated by Brill-Deser-Fadeev (cf.[BDF]), Bartnik proved that stationary metrics are critical points of the ADM energy functional on the constraint manifold. However, if the spacetime has non-empty boundary, boundary terms arise from the variation of $\mathcal H$; they were explicitly identified by Bartnik in [B1], and the boundary terms vanish if and only if the Bartnik boundary data (1.2) is preserved in the variation. 

It was conjectured in [B1] that the Bartnik quasi-local mass of a given data set $(\Omega,g,K)$ must be realized by an admissible extension $(M,g,K)$ which can be embedded as an initial data set into a stationary vacuum spacetime. To solve this conjecture, one of the well-known and fundamental open problems raised by Bartnik in [B1] is the following:
\begin{equation}
\textit{Is the Bartnik boundary data elliptic for stationary vacuum metrics?}
\end{equation}
In this paper, we give a positive answer to this question. 

A stationary spacetime $(V^{(4)},g^{(4)})$ is a 4-manifold $V^{(4)}$ with a smooth Lorentzian metric $g^{(4)}$ of signature $(-,+,+,+)$, which admits a time-like Killing vector field. In addition, a stationary spacetime is called $vacuum$ if it solves the vacuum Einstein equation
\begin{equation}
Ric_{g^{(4)}}=0.
\end{equation}

Throughout this paper, we assume that the stationary spacetime $(V^{(4)},g^{(4)})$ is \textit{globally hyperbolic}, i.e. it admits a Cauchy surface $M$ and $V^{(4)}\cong\mathbb{R}\times M$. In this case, there exists a global time function $\tau$ on $V^{(4)}$ so that $M=\{\tau=0\}$ and every surface of constant $\tau$ is a Cauchy surface. Since the spacetime is stationary, one can choose local coordinates $\{\tau, x^i\}~(i=1,2,3)$ so that ${\partial_{\tau}}$ is the time-like Killing vector field. Then the metric $g^{(4)}$ can be written globally in the form
\begin{equation}
g^{(4)}=-N^2d\tau^2+g_{ij}(dx^i+X^id\tau)(dx^j+X^jd\tau).
\end{equation}
Notice that since $\partial_{\tau}$ is a Killing vector field, the stationary spacetime $(V^{(4)},g^{(4)})$ is vacuum if and only if the equation (1.5) holds on $M$. 

\textit{Remark} In the expression of $g^{(4)}$ above, the scalar field $N$ and the vector field $X$ in $V^{(4)}$ are usually called the \textit{lapse function} and the \textit{shift vector} of $g^{(4)}$ in this 3+1 formalism of the spacetime. The tensor field $g$ is the induced (Riemannian) metric on the Cauchy surfaces $\{\tau=\text{constant}\}\subset (V^{(4)},g^{(4)})$. Since the spacetime is stationary, the fields $g, X$ and $N$ are all independent of the time variable $\tau$, so they can be regarded as tensor fields on the hypersurface $M$. Consequently, the vacuum equation (1.5) is an elliptic system (modulo gauge) of the fields $(g,X,N)$ on $M$.

Let $K$ be the second fundamental form of $M\subset (V^{(4)},g^{(4)})$. The triple $(M,g,K)$ is called an $initial$ $data$ $set$ of the spacetime. In the case where the boundary $\partial M$ is nonempty, we can impose the Bartnik boundary condition (1.3) on this data set, coupling with the vacuum equation (1.5). So we obtain a boundary value problem (BVP) as,
\begin{equation}
\begin{split}
&Ric_{g^{(4)}}=0\quad\text{on }M,\\
&\begin{cases}
g_{\partial M}=\gamma\\
H_{\partial M}=H\\
tr_{\partial M}K=k\\
\omega_{\mathbf n_{\partial M}}=\tau
\end{cases}\quad\text{on }\partial M,
\end{split}
\end{equation}
where $\gamma,H,k$, and $\tau$ are prescribed tensor fields on $\partial M$. Now, the ellipticity question (1.4) is essentially asking whether this BVP is elliptic.

Another way to formulate question (1.4) is to establish a boundary map. Let  $\mathbf B(\partial M)$ denote the space of Bartnik boundary data, i.e. space of tuples $(\gamma,H,k,\tau)$ on $\partial M$. Let $\mathcal E$ be the space of stationary vacuum metrics on $V^{(4)}$. Then a natural boundary map $\Pi_1$ arises as,
\begin{equation}
\begin{split}
\Pi_1:\mathcal E&\rightarrow\mathbf B(\partial M),\\
\Pi_1(g^{(4)})&=(g_{\partial M},H_{\partial M}, tr_{\partial M} K,\omega_{\mathbf n}).
\end{split}
\end{equation}
 The map $\Pi_1$ being Fredholm is essentially equivalent to that BVP (1.7) is elliptic. 
 
 However, it is easy to observe that equation (1.5) is not elliptic, since it is invariant under diffeomorphisms, i.e., if $g^{(4)}$ is a stationary metric that solves (1.5), then the pull back metric $\Phi^*g^{(4)}$ of $g^{(4)}$ under an arbitrary time-independent diffeomorphism $\Phi$ of $V^{(4)}$ gives another stationary vacuum solution. This means that we need to add gauge terms to the BVP (1.7), and at the same time modify the domain space $\mathcal E$ in (1.8) to a moduli space.

In this paper, we first analyze how to choose the right domain space for the boundary map to be well-defined. We conclude in \S 2 that the Bartnik boundary map should be established as,
\begin{equation*}
\begin{split}
\Pi:\mathbb E&\rightarrow\mathbf B(\partial M),\\
\Pi([g^{(4)}])&=(g_{\partial M},H_{\partial M}, tr_{\partial M} K,\omega_{\mathbf n}).
\end{split}
\end{equation*}
Here the moduli space $\mathbb E$ is the quotient of $\mathcal E$ by a particular diffeomorphism group $\mathcal D$. We refer to \S 2, cf.(2.14), for the exact definition of $\mathcal D$; roughly it is a natural intermediate group $\mathcal D_3\subset\mathcal D\subset\mathcal D_4$ between the groups of 3-dimensional diffeomorphisms on $M$ fixing the boundary $\partial M$ and 4-dimensional time-independent diffeomorphisms on $V^{(4)}$ fixing $\partial V^{(4)}$. In order to prove ellipticity of the map $\Pi$, we establish in \S3 an associated BVP under a particular technical assumption (cf. Assumption 3.1). 
We prove this BVP is elliptic in \S 4, and from this derive the main theorem of this paper:
\begin{theorem}
The moduli space $\mathbb E$ is a $C^{\infty}$ smooth Banach manifold of infinite dimension and the boundary map $\Pi$ is Fredholm.
\end{theorem}
We show in \S 5 that the theorem is still true without the technical assumption in \S 3, completing the proof of Theorem 1.1.

 To conclude, we apply this ellipticity result in \S6 to show that the Bartnik boundary data near the standard flat (Minkowski) metric $\tilde g^{(4)}_0$ on $\mathbb R\times(\mathbb R^3\setminus B^3)$ can be locally uniquely realized by a stationary vacuum metric up to diffeomorphisms in $\mathcal D$.

\begin{theorem}
There is a neighborhood $\mathcal U\subset\mathbf B(S^2)$ of the standard flat boundary data $(g_0,2,0,0)$ such that for any $(\gamma,H,k,\tau)\in\mathcal U$, there is a unique stationary vacuum metric $g^{(4)}\in\mathcal E$ near $\tilde g^{(4)}_0$ up to diffeomorphisms in $\mathcal D$, for which 
$$\Pi_1(g^{(4)})=(\gamma,H,k.\tau).$$
\end{theorem}

Throughout, we assume the hypersurface $M\cong\mathbb R^3\setminus B^3$ (exterior problem), together with certain asymptotically flat conditions on the metric $g^{(4)}$. Meanwhile, all the methods and results here can be applied equally well in the case where $M\cong B^3$ (interior problem).

This paper is a continuation of our previous paper [Az], in which we developed a general method to prove the ellipticity of boundary value problems for the stationary vacuum spacetime. Expanding this method, we study the ellipticity of the Bartnik boundary data here. Theorem 1.1 is a generalization of the results proved in [AK], where spacetimes are static. Theorem 1.2 generalizes the result in [A2] for static metrics. We refer to [Mi],[J],[R] for other existence results on stationary vacuum extensions of boundary data.

The results we prove in this paper provide a firm foundation for future work on Bartnik's conjecture about the quasi-local mass in spacetimes and the existence problem of stationary vacuum metrics that satisfy the Bartnik boundary conditions.  To the author's knowledge, this is the first ellipticity result of the Bartnik boundary data for general stationary vacuum metrics.\\
\\
\textbf{Acknowledgements}
I would like to express great thanks to my advisor Michael Anderson for suggesting this problem and for valuable discussions and comments.
\section{background}
Fix a 3-dimensional manifold $M\cong\mathbb R^3\setminus B^3$. Let $V^{(4)}\cong\mathbb R\times M$. By fixing a diffeomorphism between $V^{(4)}$ and $\mathbb R\times M$, we can equip $V^{(4)}$ with the natural coordinates $\{(t,p)\}$ where $t\in\mathbb R$ and $p\in M$. Fix the hypersurface $\{t=0\}\subset V^{(4)}$ and identify it with $M$. 
Let $\mathcal S$ denote the space of Lorenztian metrics $g^{(4)}$ on $V^{(4)}$ which satisfy the following conditions:\\
\\
\textit{1. (globally hyperbolic) The metric $g^{(4)}$ can be expressed globally as
\begin{equation}
g^{(4)}=-N^2dt^2+g_{ij}(dx^i+X^idt)(dx^j+X^jdt),
\end{equation}
where $\{x^i\}(i=1,2,3)$ are local coordinates on $M$. There exists some time independent function $f\in C^{m,\alpha}_{\delta}(V^{(4)})$ (cf.\S7.1) so that the transformation $\tau=t+f$ makes $\tau$ a time function of the spacetime $(V^{(4)},g^{(4)})$ in the sense of general relativity and every surface of constant $\tau$ is a Cauchy surface (cf. equation (1.6)). In particular, if $f\equiv0$ then expression (2.1) and (1.6) agree. \\
2. (stationary) The vector field $\partial_t$ is a time-like Killing vector field in $(V^{(4)},g^{(4)})$. So the triple $(g,X,N)$ is independent of $t$ and can be regarded as tensor fields on $M$ (cf. the remark below equation (1.6)). In addition, since $\langle\partial_t,\partial_t\rangle_{g^{(4)}}=-N^2+||X||_{g}^2<0$, one has \begin{equation}N^2>||X||_{g}^2.\end{equation}
3. (asymptotically flat) The metric $g^{(4)}$ decays to the flat (Minkowski) metric at infinity. Explicitly, $N$, $X$ and $g$ belong to the weighted H\"older spaces on $M$, given by,
\begin{equation}
\begin{split}
&g\in Met^{m,\alpha}_{\delta}(M),\\
&N-1\in C^{m,\alpha}_{\delta}(M),\\
&X\in T^{m,\alpha}_{\delta}(M),
\end{split}
\end{equation}
for some fixed number $m\geq 2,~ 0<\alpha<1,$ and $\frac{1}{2}<\delta<1$. We refer to the Appendix \S 7.1 for the precise definition of the weighted H\"older spaces.\\
 }$~~$
 
 Throughout this paper, we will use $\langle X,Y \rangle_{g}$ to denote the inner product of two vector fields with respect to the metric $g$. The (square) norm of a vector field $X$ with respect to the metric $g$ is $\langle X,X\rangle_{g}=||X||^2_{g}$. We will omit the metric in the subcript when it is clear in the context what metric is being used. 
 \begin{remark} Based on the definition of the space $\mathcal S$, it is easy to observe that $\mathcal S$ is invariance under the action of the diffeomorphism group $\mathcal D_4$ (cf.(2.6) below). Consequently, the tensor field $g$ in (2.1), which can be taken as the induced metric on the hypersurface $M\subset (V^{(4)},g^{(4)})$, is not necessarily Riemannian.
 \end{remark}
It is obvious that an element in $\mathcal S$ is uniquely determined by a triple of fields $(g,X,N)$ on $M$. Thus $\mathcal S$ is an open domain in a Banach space and so admits smooth Banach manifold structure.

As is mentioned in the introduction, one can establish BVP (1.7) for $g^{(4)}\in\mathcal S$, but in order to make it elliptic, we need to add gauge terms. A standard choice is to use the Bianchi gauge, leading to a modified system with unknown $g^{(4)}\in\mathcal S$ as follows:
 \begin{equation}
	\begin{split}
		&Ric_{g^{(4)}}+\delta^*_{g^{(4)}}\beta_{\tilde g^{(4)}}g^{(4)}=0
		\quad\text{on}\quad M,\\
		&\begin{cases}
			g_{\partial M}=\gamma\\
			H_{\partial M}=H\\
			tr_{\partial M}K=k\\
			\omega_{\mathbf n}=\tau\\
			\beta_{\tilde g^{(4)}}g^{(4)}=0.
		\end{cases}\quad\text{on}\quad\partial M
	\end{split}
\end{equation}
In the system above, we add the term $\delta^*_{g^{(4)}}\beta_{\tilde g^{(4)}}g^{(4)}$ in the vacuum equation, and add the Dirichlet condition of the gauge term $\beta_{\tilde g^{(4)}}g^{(4)}$ on the boundary.  Here the gauge term $\beta_{\tilde g^{(4)}}g^{(4)}$ is the Bianchi operator acting on the metric $g^{(4)}$ with respect to a fixed stationary vacuum metric $\tilde g^{(4)}$, i.e. $\beta_{\tilde g^{(4)}}g^{(4)}=\delta_{\tilde g^{(4)}}g^{(4)}+\frac{1}{2}dtr_{\tilde g^{(4)}}g^{(4)}$, where the reference metric $\tilde g^{(4)}\in\mathcal E$ (cf.(2.5) below).  We use $\delta$ to denote the divergence operator $\delta =-tr\nabla $, and $\delta^*$ denotes the formal adjoint of the divergence operator, i.e. $\delta^*_{g^{(4)}}Y=\frac{1}{2}L_{Y}g^{(4)}$ for any $Y\in TV^{(4)}$. Among the Bartnik boundary conditions, here and throughout the following, we use $\omega_{\mathbf n}$ as the abbreviation of $\omega_{\mathbf n_{\partial M}}$.

The effect of adding the gauge term to BVP (1.7) as above is to give a slice to the action on the solution space of (1.7) by the group $\mathcal D_4$ (cf.(2.6) below) of diffeomorphisms of the spacetime fixing the boundary $\partial V^{(4)}$. However, such a modification has two issues. First, it is easy to observe that $(2.4)$ is not well posed, because there are 10 interior equations on $M$ but 11 boundary conditions on $\partial M$ --- notice that, the gauge term $\beta_{\tilde g^{(4)}}g^{(4)}$ defines a vector field in $V^{(4)}$, so it contributes $4$ extra boundary equations in $(2.4)$. 

Secondly, the associated boundary map to BVP (2.4) is not well-defined. Let $\mathcal E$ be the space of stationary vacuum metrics, i.e.
\begin{equation}
\mathcal E=\{g^{(4)}\in\mathcal S:~Ric_{g^{(4)}}=0\}.
\end{equation}
As is explained above, after adding the gauge term $\beta_{\tilde g^{(4)}}g^{(4)}$, the boundary map $\Pi_1$ defined in (1.8) should be modified to $\Pi_2$ as follows, 
\begin{equation*}
\begin{split}
\Pi_2:&\mathcal E/\mathcal D_4\rightarrow\mathbf B(\partial M),\\
\Pi_2([g^{(4)}])&=(g_{\partial M}, H_{\partial M}, tr_{\partial M}K,\omega_{\mathbf n}),
\end{split}
\end{equation*}
where the target space $\mathbf B(\partial M)$ is given by $\mathbf B(\partial M)=Met^{m,\alpha}(\partial M)\times [C^{m-1,\alpha}(\partial M)]^2\times \wedge_1^{m-1,\alpha}(\partial M)$ (cf.\S 7.1 for the notations of various spaces of tensor fields).
However, this map is not well defined, because elements in $\mathcal D_4$ do not always preserve the Bartnik boundary data (cf. Proposition 2.2 below), which means that the Bartnik boundary data is not well defined for an element $[g^{(4)}]$ --- an equivalence class of metrics --- in the moduli space $\mathcal E/\mathcal D_4$. 

Since we are working with stationary metrics, it is natural to require elements in $\mathcal D_4$ to be time-independent and preserve the Killing vector field $\partial_t$. Thus a general element in $\mathcal D_4$ can be decomposed into two parts --- a diffeomorphism on the hypersurface $M$ and a translation of time, i.e. $\mathcal D_4$ can be defined as,
\begin{equation}
\begin{split}
\mathcal D_4=\{\Phi_{(\psi,f)}|~&\psi\in D^{m+1,\alpha}_{\delta}(M)\text{ and }\psi|_{\partial M}=Id_{\partial M};\\
&~f \in C^{m+1,\alpha}_{\delta}(M)\text{ and } f|_{\partial M}=0;\\
&\Phi_{(\psi,f)}:V^{(4)}\rightarrow V^{(4)},\\
&\Phi_{(\psi,f)}[t,p]=[t+f,\psi(p)],\quad\forall t\in\mathbb R,~p\in M.~\},
\end{split}
\end{equation}
Here $D^{m+1,\alpha}_{\delta}(M)$ denotes the group of $C^{m+1,\alpha}$ diffeomorphisms of $M$ which are asymptotically $Id_{M}$ at the rate of $\delta$ (cf.\S7.1).
\begin{proposition} If an element $\Phi_{(\psi,f)}\in\mathcal D_4$ has a nontrivial time translation function $f$, then it does not preserve the Bartnik boundary data on $\partial M$.
\end{proposition}
\begin{proof}
Equip $M=\{t=0\}$ with local coordinates $\{x^i\},~i=1,2,3.$ Choose a function $f\in C^{m+1,\alpha}_{\delta}(M)$, and take the diffeomorphism $\Phi_{(Id_M,f)}\in\mathcal D_4:$
\begin{equation*}
\begin{split}
\Phi_{(Id_M,f)}: V^{(4)}&\rightarrow V^{(4)}\\
\Phi_{(Id_M,f)}(t,x_1,x_2,x_3)&=(t+f,x_1,x_2,x_3).
\end{split}
\end{equation*}
In the following, we will use $\Phi_f$ as the abbreviation of $\Phi_{(Id_M,f)}$.
Take a spacetime metric $g^{(4)}\in\mathcal S$ expressed as,
\begin{equation*}
g^{(4)}=-N^2dt^2+g_{ij}(dx^i+X^idt)(dx^j+X^jdt).
\end{equation*}
Let $\hat g^{(4)}$ denotes the pull back metric, i.e. $\hat g^{(4)}=\Phi_f^*g^{(4)}$. Then we have,
\begin{equation*}
	\begin{split}
		\hat g^{(4)}&=-N^2[d(t+f)]^2+g_{ij}[dx^i+X^id(t+f)][dx^j+X^jd(t+f)]\\
		&=-u^2dt^2-u^2df\odot dt+X_idx^i\odot dt-u^2(df)^2+X_idx^i\odot df+g_{ij}dx^idx^j,
	\end{split}
\end{equation*}
where $u^2=N^2-|X|_g^2$. Here we use $\odot$ to denote the symmetrized product between two 1-forms $A,B$, i.e. $A\odot B=A\otimes B+B\otimes A$.
From the expression above, one easily observes that the induced metric on $M\subset (V^{(4)}, \hat g^{(4)})$ is given by,
\begin{equation}
\hat g=-u^2(df)^2+X_idx^i\odot df+g_{ij}dx^idx^j.
\end{equation}
Plugging $f|_{\partial M}=0$ in the equation above, it is obvious that the first Bartnik boundary term $g_{\partial M}$ in (2.4) remains the same under such a time translation. However, this is not the case for the other data $H_{\partial M}$, $tr_{\partial M}K$ and $\omega_{\mathbf n_{\partial M}}$.

Let $\mathbf{N}$ denote the future-pointing time-like unit normal vector to the slice $M\subset (V^{(4)},g^{(4)})$ and $\mathbf n$ denote the outward unit normal of $\partial M\subset (M, g)$. Define $(\mathbf{\hat N},\mathbf{\hat n})$ in the same way for $M\subset (V^{(4)},\hat g^{(4)})$. Then on the boundary $\partial M$, the pairs $(\mathbf N,\mathbf n)$ and $(\mathbf{\hat N},\mathbf{\hat n})$ are related in the following way,
$$
\begin{bmatrix}
d\Phi_f(\mathbf{\hat N})\\
d\Phi_f(\mathbf{\hat n})
\end{bmatrix}
=
\begin{bmatrix}
a&b\\
b&a
\end{bmatrix}
\begin{bmatrix}
\bf N\\
\bf n
\end{bmatrix},
$$
where $a,b$ are scalar fields on $\partial M$ and $a^2-b^2=1$. We refer to \S7.2 for the detailed proof.

Let $\nabla$ be the Levi-Civita connection of the spacetime $(V^{(4)},g^{(4)})$, and $\hat{\nabla}$ denotes that of the spacetime $(V^{(4)},\hat g^{(4)})$, then $\hat{\nabla}=\Phi_f^*(\nabla)$. We use $H_{\partial M}$, $tr_{\partial M}K$ and $\omega_{\mathbf n}$ to denote the Bartnik boundary data of $(V^{(4)},g^{(4)})$ on $\partial M$; and use $\hat H_{\partial M}$, $tr_{\partial M}\hat K$ and $\hat\omega_{\hat{\mathbf n}}$  as that of $(V^{(4)},\hat g^{(4)})$. Then we have the following formula for the mean curvature:
\begin{equation}
\begin{split}
\hat H_{\partial M}&=tr_{\partial M}(\hat\nabla\mathbf{\hat n})\\
&=tr_{\partial M}[\nabla d\Phi_f(\mathbf{\hat n})]\\
&=tr_{\partial M}[\nabla(b\mathbf{N}+a\mathbf{n})]\\
&=btr_{\partial M}(\nabla \mathbf{N})+atr_{\partial M}(\nabla\mathbf{n})\\
&=btr_{\partial M}K+aH_{\partial M}.
\end{split}
\end{equation}
It is easy to show that $tr_{\partial M}K$ is transformed in a similar way as above, i.e.
\begin{equation}
tr_{\partial M}\hat K=atr_{\partial M}K+bH_{\partial M}.
\end{equation}
As for the last boundary term $\omega_{\mathbf n}$, one has $\forall v\in T(\partial M)$,
\begin{equation*}
\begin{split}
\hat \omega_{\hat{\mathbf n}}(v)&=\hat K(\hat{\mathbf n},v)=\langle\hat\nabla_v\hat{\mathbf N},\hat{\mathbf n}\rangle_{\hat g^{(4)}}\\
&= \langle\Phi_f^*(\nabla)_{v}\mathbf{\hat N},\mathbf{\hat n}\rangle_{\Phi_f^*g^{(4)}}\\
&=\langle\nabla_{d\Phi_f(v)}(a\mathbf{N}+b\mathbf n),~b\mathbf{N}+a\mathbf n\rangle_{g^{(4)}}\\
&=-b\cdot \nabla_{d\Phi_f(v)}a+a\cdot \nabla_{d\Phi_f(v)}b+(a^2-b^2)\langle\nabla_{d\Phi_f(v)} \mathbf N,\mathbf n\rangle_{g^{(4)}}\\
&=a^2\nabla_{d\Phi_f(v)}(b/a)+K(\mathbf n,d\Phi_f(v)),\\
&=a^2v(b/a)+\omega_{\mathbf n}(v).
\end{split}
\end{equation*}
Here the last equality is based on the observation that $d\Phi_f(v)=v~\forall v\in T(\partial M)$, since $\Phi_f|_{\partial M}=Id_{\partial M}$. From the formula above, we conclude that,
\begin{equation}
\hat \omega_{\mathbf {\hat n}}=a^2d_{\partial M}(b/a)+\omega_{\mathbf n},
\end{equation}
where $d_{\partial M}(b/a)$ denotes the exterior derivative of the scalar field on $\partial M$. Along the boundary $\partial M$, one has
\begin{equation}
a=\frac{1+\langle X,\mathbf n\rangle\mathbf n(f)}{\sqrt{[1+\langle X,\mathbf n\rangle\mathbf n(f)]^2-N^2|\mathbf n(f)|^2}}.
\end{equation}
We refer to the Appendix \S 7.2 for the detailed calculation of the scalar fields $a,b$.
Therefore, if the function $f$ is nontrivial, in the sense that $\mathbf n(f)|_{\partial M}\neq0$ and $a\neq1$ in (2.11), then it is easy to observe from equations (2.8-10) that the Bartnik boundary conditions are not invariant under the diffeomorphism $\Phi_f$.
\end{proof}

In view of the fact above, one may suggest to reduce the diffeomorphism group $\mathcal D_4$ in the definition of the boundary map to a smaller one $\mathcal D_3$  consisting of only 3-dim diffeomorphism on the slice, i.e. 
\begin{equation}
\mathcal D_3=\{\Phi_{(\psi,f)}\in\mathcal D_4:~f\equiv 0 \text{ on }M\}.
\end{equation}
However, this approach does not work either. Let $\Pi_3$ be the associated boundary map as follows,
$$\Pi_3:\mathcal E/\mathcal D_3\rightarrow \mathbf B$$
$$\Pi_3([g^{(4)}])=(g_{\partial M},H_{\partial M},tr_{\partial M}K,\omega_{\mathbf n}).$$
Given a fixed boundary condition $(\gamma,H,k,\tau)$, and an element $g^{(4)}$ in the pre-image set $\Pi_3^{-1}[(\gamma,H,k,\tau)]$, we can take an arbitrary function $f\in C^{m+1,\alpha}_{\delta}(M)$ such that $f|_{\partial M}=\mathbf n(f)|_{\partial M}=0$, and make time translation $\Phi_f$ to obtain a new metric $\bar g^{(4)}=\Phi_f^*g^{(4)}$. Then, by the previous analysis, $\bar g^{(4)}$ also belongs to $\Pi_3^{-1}[(\gamma,H,k,\tau)]$. 

By taking a smooth curve (parametrized by $\tau$) of such time translations $\Phi_{f(\tau)}$ $(\tau\in (-1,1))$, we get a family of metrics $g^{(4)}(\tau)=\Phi^*_{f(\tau)}g^{(4)}$. Let $f_{\tau}=\frac{\partial}{\partial\tau}|_{\tau=0}f$, then the infinitesimal deformation of the spacetime metric at $\tau=0$ is of the form,
\begin{equation*}
(g^{(4)})'=L_{f_{\tau}\partial_t}g^{(4)}= df_{\tau}\odot (\partial t)^{\flat}=df_{\tau}\odot (-u^2dt+X_idx^i).
\end{equation*}
By construction, $g^{(4)}(\tau)\in\Pi_3^{-1}[(\gamma,H,k,\tau)]$ for all $\tau$, which implies that $(g^{(4)})'\in \text{Ker} D\Pi_3$. Such a kernel element is nontrivial if it is not tangent to any 3-dim diffeomorphism variation, i.e. the following equation is not solvable for $Z\in T^{m,\alpha}_{\delta}M$,
\begin{equation}
df_{\tau}\odot (-u^2dt+X_idx^i)= L_Zg^{(4)}.
\end{equation}
Since $(2.13)$ is an overdetermined system for $Z$, it is not solvable for generic choices of $f_{\tau}$. This means that the kernel of $D\Pi_3$ should be of infinite dimension, which indicates that $\Pi_3$ is not a Fredholm map.

From all the previous analysis, we notice that the Neumann data $\mathbf n(f)$ of the time translation function plays an important role in choosing the right diffeomorphism group. This suggests defining a new group $\mathcal D$ as,
\begin{equation}
\mathcal D=\{\Phi_{(\psi,f)}\in\mathcal D_4:\mathbf n_g(f)=0\text{ on }\partial M\}.
\end{equation}
It is in fact an intermediate group in the sense that $\mathcal D_3\subset \mathcal D\subset\mathcal D_4$.
\medskip
\begin{remark} The vector field $\mathbf n_g$ in (2.14) can be taken as the unit normal vector of $\partial M$ with respect to any Riemannian metric $g$ on $M$ --- the group $\mathcal D$ does not depend on the choice of the metric $g$. In fact, it is easy to observe that $\mathcal D$ can be defined in an equivalent way:
\begin{equation*}
\mathcal D=\{\Phi_{(\psi,f)}\in\mathcal D_4: df=0\text{ at }\partial M\}.
\end{equation*}
Notice that $\mathbf n(f)=0$ in (2.14) yields $a=1$ in (2.11). This further implies that, geometrically, elements in the group $\mathcal D$ are diffeomorphisms of the spacetime $(V^{(4)},g^{(4)})$ which fix the boundary $\partial M$ and the time-like unit normal vector field $\mathbf N$ along $\partial M$.
\end{remark}
\medskip

Now, define $\mathbb E$ to be the quotient space,
$$\mathbb E=\mathcal E/\mathcal D.$$
Elements in $\mathbb E$ are equivalence classes $[g^{(4)}]$ given by,
\begin{equation*}
\begin{split}
[g^{(4)}]=\{\Phi_{(\psi,f)}^*g^{(4)}:~ g^{(4)}\in\mathcal E,~ \Phi_{(\psi,f)}\in\mathcal D\}.
\end{split}
\end{equation*}
Now we can consider the natural boundary map:
\begin{equation}
       \begin{split}
       		 \Pi:&\mathbb E\rightarrow\mathbf B\\
		 \Pi([g^{(4)}])&=(g_{\partial M},H_{\partial M},tr_{\partial M}K,\omega_{\mathbf n}).
	\end{split}
\end{equation}

This map is well defined --- the Bartnik boundary data is the same for all the metrics inside one equivalence class $[g^{(4)}]\in\mathbb E$, because the transformation formulas $(2.8-10)$ show that Bartnik boundary data is preserved under diffeomorphisms in $\mathcal D$. In the following sections we will prove this boundary map $\Pi$ is Fredholm.
\begin{remark} By the Remark 2.1, a general spacetime metric $g^{(4)}$ in an equivalence class $[g^{(4)}]\in\mathbb E$ does not necessarily induce Rimannian geometry on the fixed slice $M$, which might make the Bartnik boundary data not well-defined on the slice $M\subset (V^{(4)},g^{(4)})$. To solve this problem, we can restrict the boundary map $\Pi$ in (2.15) to the open subset $\mathbb E'\subset \mathbb E$, where every equivalent class $[g^{(4)}]\in\mathbb E'$ admits a representative $g^{(4)}$ which induces Riemannian metric on $M$. On the other hand, since the ellipticity of the Bartnik boundary data is only a local property, we can focus the study in a neighborhood of such metric $g^{(4)}\in\mathcal E$ and think of $\mathbb E$ locally as a slice of $\mathcal E$ under the action of $\mathcal D$. The boundary map $\Pi$ being Fredholm implies that the linearization $D\Pi_1$ of the map $\Pi_1$ in (1.8) at $g^{(4)}$ has finite dimensional kernel transverse to the action of $\mathcal D$ in the sense that 
$\text{Ker}D\Pi_1/\sim$ has finite dimension. Here two deformations $h^{(4)}_1,h^{(4)}_2\in \text{Ker}D\Pi_1$ are equivalent ($h^{(4)}_1\sim h^{(4)}_2$) if the difference $(h^{(4)}_1-h^{(4)}_2)=\delta^*_{g^{(4)}}Y$ for some $Y\in T\mathcal D$. In the following sections we always assume that the reference spacetime metric $\tilde g^{(4)}$ is chosen so that the slice $M$ is a Cauchy surface in $(V^{(4)},\tilde g^{(4)})$.
\end{remark}

\section{The well-defined BVP}
Throughout this section, we take $\tilde g^{(4)}\in\mathcal E$ as a fixed reference metric and make the following assumption:
\begin{assumption}
The BVP with unknown $Y\in T^{m,\alpha}_{\delta}(V^{(4)})$, given by,
\begin{equation}
	\begin{cases}
		\beta_{g^{(4)}}\delta_{g^{(4)}}^*Y=0\quad\text{on}\quad M\\
		Y=0\quad\text{on}\quad\partial M
	\end{cases}
\end{equation}
has only the zero solution $Y=0$ when $g^{(4)}=\tilde g^{(4)}$.
\end{assumption}
In the above, $T^{m,\alpha}_{\delta}(V^{(4)})$ denotes the space of $C^{m,\alpha}$ vector fields in $V^{(4)}$, which are asymptotically zero at the rate of $\delta$ and in addition time-independent(cf.\S7.1). Throughout this paper, we say a tensor field $T$ in $V^{(4)}$ is \textit{time-independent} if $L_{\partial_t}T=0$. 

In the following, we call the operator $\beta_{\tilde g^{(4)}}\delta^*_{\tilde g^{(4)}}$ \emph{invertible} if Assumption 3.1 holds. Note that since the system (3.1) is elliptic and self-adjoint (cf.\S5.1), it has trivial kernel if and only if the following map is an isomorphism:
\begin{equation*}
\beta_{\tilde g^{(4)}}\delta^*_{\tilde g^{(4)}}: \mathcal T\to T^{m-2,\alpha}_{\delta+2}(V^{(4)}),
\end{equation*}
where the domain space $\mathcal T=\{Y\in T^{m,\alpha}_{\delta}(V^{(4)}): Y=0\text{ on }\partial M\}$. This is an open condition. Thus if $\beta_{\tilde g^{(4)}}\delta^*_{\tilde g^{(4)}}$ is invertible, then so is the operator $\beta_{g^{(4)}}\delta^*_{ g^{(4)}}$ for $g^{(4)}$ near $\tilde g^{(4)}$ in the space $\mathcal S$. We refer to Remark 4.3 for more insights about the Assumption 3.1.

Based on the discussions in \S2, we modify the system (2.4) to a new BVP with unknowns $(g^{(4)},F)\in  \mathcal S\times C^{m,\alpha}_{\delta}(M)$ as follows,
\begin{equation}
	\begin{split}
		&\quad\quad\begin{cases}
			Ric_{g^{(4)}}+\delta^*_{g^{(4)}}\beta_{\tilde g^{(4)}}g^{(4)}=0\\
			\Delta F=0
		\end{cases}
		\quad\text{on}\quad M,\\
		&\quad\quad\begin{cases}			
			g_{\partial M}=\gamma\\
			aH_{\partial M}+btr_{\partial M}K=H\\
			atr_{\partial M}K+bH_{\partial M}=k\\
			\omega_{\mathbf n}+a^2d_{\partial M}(a/b)=\tau\\
			\beta_{\tilde g^{(4)}}g^{(4)}=0
		\end{cases}\quad\text{on}\quad\partial M,
	\end{split}
\end{equation}
where 
\begin{equation}
a=\frac{1+\langle X,\mathbf n\rangle F}{\sqrt{(1+\langle X,\mathbf n\rangle F)^2-N^2F^2}},~\text{ and }~b=\sqrt{a^2-1},
\end{equation}
with $N$ and $X$ denoting the lapse function and shift vector of $g^{(4)}$ as in (2.1). In the second equation above, $\Delta=-trHess$ denotes the Laplace operator with respect to the metric $g^{(4)}$. Here we think of $F\in C^{m,\alpha}_{\delta}(M)$ as a function in $V^{(4)}$ by expanding it time-independently. The argument to follow works in the same way if one sets $\Delta$ to be the Laplacian of the induced Riemannian metric $g$ on the slice $M$. But with the former choice, the principal symbol which we will compute in $\S 4$ is simpler.

Applying the Bianchi operator to the first equation of $(3.2)$, one obtains,
\begin{equation}
\beta_{g^{(4)}}\delta^*_{g^{(4)}}[\beta_{\tilde g^{(4)}}g^{(4)}]=0\quad\text{on }M.
\end{equation}
In addition, the last boundary condition in (3.2) gives,
\begin{equation}\beta_{\tilde g^{(4)}}g^{(4)}=0\quad\text{on } \partial M.
\end{equation}
Combining (3.4) and (3.5), together with the Assumption 3.1, it follows that,
\begin{equation*}
\begin{split}
\beta_{\tilde g^{(4)}}g^{(4)}=0,\quad\forall \text{ solution } g^{(4)}\text{ of (3.2) near }\tilde g^{(4)}.
\end{split}
\end{equation*}
Therefore, if $g^{(4)}$ is a solution to (3.2) near $\tilde g^{(4)}$, then it must be Ricci flat and Bianchi-free with respect to $\tilde g^{(4)}$. So we define a solution space $\mathcal C$ as follows:
\begin{equation*}
	\begin{split}
		\mathcal C:=\{~(g^{(4)},F)\in\mathcal S\times C^{m,\alpha}_{\delta}(M):~&Ric_{g^{(4)}}=0,~\beta_{\tilde g^{(4)}}g^{(4)}=0,~\Delta F=0~\text{on}~M~\}.
			\end{split}
\end{equation*}
Obviously, $(\tilde g^{(4)},0)\in\mathcal C$. Let $\tilde \Pi$ be the boundary map:
\begin{equation*}
	\begin{split}
		\tilde \Pi:~\mathcal C&\rightarrow\mathbf{B}\\
		\tilde \Pi(g^{(4)},F)=(g_{\partial M},a H_{\partial M}+btr_{\partial M} K,&atr_{\partial M} K+b H_{\partial M},\omega_{\mathbf n}+a^2d_{\partial M}(b/a)).
	\end{split}
\end{equation*}
This map $\tilde\Pi$ is closely related to the boundary map $\Pi$ defined in (2.15) near the reference pair $(\tilde g^{(4)},0)\in\mathcal C$. In fact, we have the following theorem.
\begin{theorem}
There is a map $\mathcal P:\mathcal C\to\mathbb E$ which is locally a diffeomorphism near $(\tilde g^{(4)},0)$, and the boundary maps $\Pi$ and $\tilde\Pi$ are related by
\begin{equation*}
\tilde\Pi=\Pi\circ\mathcal P.
\end{equation*}
\end{theorem}
\begin{proof}Given an element $(\hat g^{(4)},\hat F)\in\mathcal C$, one can take a function $f$ on $M$ such that $f|_{\partial M}=0$ and $\mathbf n(f)|_{\partial M}=\hat F|_{\partial M}$, and apply the diffeomorphism $\Phi_{(\psi,f)}\in\mathcal D_4$ to $\hat g^{(4)}$ where $\psi$ is an arbitrary diffeomorphism in $\mathcal D^{m,\alpha}_{\delta}(M)$ with $\psi|_{\partial M}=Id_{\partial M}$. Thus, any element $(\hat g^{(4)},\hat F)\in\mathcal C$ gives rise to a class of elements as follows,
\begin{equation}
\{\Phi_{(\psi,f)}^*(\hat g^{(4)}):~\Phi_{(\psi,f)}\in\mathcal D_4,~\mathbf n( f)|_{\partial M}=\hat F|_{\partial M}\}.
\end{equation} 
It is easy to observe that the equivalence class above actually defines an element in $\mathbb E$. Henceforth we can define a map $\mathcal P$ as, 
\begin{equation*}
\begin{split}
\mathcal P:\mathcal C&\rightarrow\mathbb E,\\
\mathcal P(\hat g^{(4)},\hat F)&=[g^{(4)}],
\end{split}
\end{equation*}
where $[g^{(4)}]$ is defined as the equivalence class $(3.6)$.

On the other hand, consider the following map:
\begin{equation*}
\begin{split}
&\mathcal G: \mathcal S\times \mathcal D_4\rightarrow (\wedge_1)^{m,\alpha}_{\delta} (V^{(4)})\\
&\mathcal G(g^{(4)}, \Phi)=\beta_{\tilde g^{(4)}}\Phi^*g^{(4)},
\end{split}
\end{equation*}
where $(\wedge_1)^{m,\alpha}_{\delta} (V^{(4)})$ denotes the space of $C^{m,\alpha}$ 1-forms in $V^{(4)}$ which are time-independent and asymptotically zero at the rate of $\delta$(cf.\S7.1).
The linearization of $\mathcal G$ at $(\tilde g^{(4)}, Id_{V^{(4)}})$ is given by,
\begin{equation*}
\begin{split}
&D\mathcal G|_{(\tilde g^{(4)},Id_{V^{(4)}})}: T\mathcal S\times T\mathcal D_4\rightarrow (\wedge_1)^{m,\alpha}_{\delta}V^{(4)}\\
&D\mathcal G|_{(\tilde g^{(4)},Id_{V^{(4)}})}[(h^{(4)},Y)]=\beta_{\tilde g^{(4)}}\delta_{\tilde g^{(4)}}^*Y+\beta_{\tilde g^{(4)}}h^{(4)}. 
\end{split}
\end{equation*}
By the definition of $\mathcal D_4$, the vector field $Y\in T\mathcal D_4$ is time-independent, asymptotically zero and $Y=0$ on $\partial M$. So the operator $\beta_{\tilde g^{(4)}}\delta^*_{\tilde g^{(4)}}$ in the linearization above is invertible by the Assumption 3.1. Therefore, by the implicit function theorem, there is a neighborhood $U_{\tilde g^{(4)}}$ of $\tilde g^{(4)}$ in $\mathcal S$ such that for any $g^{(4)}\in U$, there is a unique element $\Phi_{(\psi,f)}\in\mathcal D_4$ near $Id_{V^{4}}$ such that the pull back metric $\Phi_{(\psi,f)}^*g^{(4)}$ is gauge-free, i.e. $\beta_{\tilde g^{(4)}}(\Phi_{(\psi,f)}^*g^{(4)})=0$ in $V^{(4)}$. Moreover, if $g^{(4)}$ is vacuum, i.e. $g^{(4)}\in U_{\tilde g^{(4)}}\cap\mathcal E$, then the gauge-free metric $\hat g^{(4)}=\Phi_{(\psi,f)}^*g^{(4)}$ is also vacuum. 

Trivially it follows,
$$g^{(4)}=(\Phi^*_{(\psi,f)})^{-1}\hat g^{(4)}=\Phi_{(\psi^{-1},-f)}^*\hat g^{(4)}.$$
Let $\hat F\in C^{m,\alpha}_{\delta}(M)$ be the unique harmonic function (with respect to the metric $\hat g^{(4)}$) on $M$ satisfying the Dirichlet boundary condition $\hat F=\mathbf n(-f)$ on $\partial M$. Since the diffeomorphism $\Phi_{(\psi,f)}$ is near $Id_{V^{(4)}}$, $\mathbf n(f)$ is close to zero. Thus $F$ is also near the zero function on $M$. Pairing it with $\hat g^{(4)}$, we obtain an element $(\hat g^{(4)},\hat F)\in\mathcal C$ near $(\tilde g^{(4)},0)$.

In addition, if two elements $ g_1^{(4)},~ g_2^{(4)}\in U_{\tilde g^{(4)}}\cap\mathcal E$ are equivalent under some diffeomorphism $\Phi_{(\psi_0,f_0)}\in\mathcal D_4$, then they correspond to the same gauge-free metric $\hat g^{(4)}$ because of the uniqueness shown above. In addition, since the time translation $f_0$ makes $\mathbf n(f_0)=0$ on $\partial M$, $g_1^{(4)}$ and $g_2^{(4)}$ also generate the same harmonic function $\hat F$ as described above. Therefore, in the neighborhood $U_{\tilde g^{(4)}}$ of $\tilde g^{(4)}$ in $\mathcal S$, all the metrics that belong to the same equivalence class $[g^{(4)}]\in\mathbb E$ give rise to a unique pair $( \hat g^{(4)},\hat F)\in\mathcal C$ near $(\tilde g^{(4)},0)$. Define $\mathbb U_{\tilde g^{(4)}}$ to be the corresponding neighborhood of $[\tilde g^{(4)}]$ in $\mathbb E$, i.e.
\begin{equation*}\begin{split}
\mathbb U_{\tilde g^{(4)}}=\{[g^{(4)}]\in\mathbb E:~\text{The equivalence class }[g^{(4)}]\text{ admits a representative } g_0^{(4)}&\\\text{ such that }g_0^{(4)}\in U_{\tilde g^{(4)}}\}&.
\end{split}\end{equation*}
Then for any $[g^{(4)}]\in\mathbb U_{\tilde g^{(4)}}$ we can take an arbitrary representative $g^{(4)}\in U_{\tilde g^{(4)}}$ and then obtain a pair $(\hat g^{(4)},\hat F)$ in the manner described above. Moreover, as is shown, the pair $(\hat g^{(4)},\hat F)$ does not depend on the representative we choose in $U_{\tilde g^{(4)}}$. In this way, one can establish a map $\mathcal {\tilde P}$ locally defined in $\mathbb E$ near $[\tilde g^{(4)}]$ and mapping $\mathbb U_{\tilde g^{(4)}}$ to a neighborhood of  $(\tilde g^{(4)},0)$ in $\mathcal C$, given by,
\begin{equation*}
\begin{split}
\mathcal{\tilde P}:\mathbb U_{\tilde g^{(4)}}&\rightarrow\mathcal C,\\
\mathcal{\tilde P}([g^{(4)}])&=(\hat g^{(4)},\hat F).
\end{split}
\end{equation*}

It is easy to check that $\mathcal P$ and $\mathcal {\tilde P}$ are the inverse map of each other near $(\tilde g^{(4)},0)$. Thus, the spaces $\mathcal C$ and $\mathbb E$ are locally diffeomorphic via $\mathcal P$.

Moreover, based on the formulas $(2.8-10)$, one can easily observe that if $[g^{(4)}]=\mathcal P(\hat g^{(4)},\hat F)$, then their Bartnik boundary data are related in the following way,
\begin{equation*}
\begin{split}
(g_{\partial M}, &H_{\partial M}, tr_{\partial M} K, \omega_{\mathbf n})\\
&=(\hat g_{\partial M},a\hat H_{\partial M}+btr_{\partial M}\hat K,atr_{\partial M}\hat K+b\hat H_{\partial M},\hat\omega_{\mathbf n}+a^2d_{\partial M}(b/a)),
\end{split}
\end{equation*}
where $a,b$ are given by the formulas in $(3.3)$ inside which $F=\hat F$ and $X,N$ are the lapse and shift of $\hat g^{(4)}$. Therefore, near $(\tilde g^{(4)},0)$ the boundary maps $\tilde \Pi$ and $\Pi$ are related by,
$$\tilde \Pi= \Pi\circ\mathcal P.$$
\end{proof}
From the analysis above, we see that locally the solution space $\mathcal C$ is a coordinate chart of the moduli space $\mathbb E$ near the reference metric class $[\tilde g^{(4)}]$, and the map $\tilde\Pi$ is the Bartnik boundary map $\Pi$ expressed in this local chart. In the following, we show that the space $\mathcal C$ admits Banach manifold structure. 
\begin{theorem}
 The space $\mathcal C$ admits smooth Banach manifold structure near $(\tilde g^{(4)},0)$.
\end{theorem}
\begin{proof}
For any stationary vacuum metric $g^{(4)}$, define $\mathcal H_{g^{(4)}}$ as the space of harmonic functions on $M$:
\begin{equation*}
\mathcal H_{g^{(4)}}=\{f\in C^{m,\alpha}_{\delta}(M):\Delta_{g^{(4)}} f=0 \text{ on } M\}.
\end{equation*}
Since $\Delta_{g^{(4)}}$ is invertible when subjected to Dirichlet boundary conditions, it is easy to prove that,
$$\mathcal H_{g^{(4)}}\cong C^{m,\alpha}(\partial M).$$
Thus $\mathcal H_{g^{(4)}}$ is a smooth Banach manifold. 

Let $P:\mathcal C\to \mathcal E$ be the projection $P(\hat g^{(4)},\hat F)=\hat g^{(4)}$. We observe that if $(\hat g^{(4)},\hat F)\in\mathcal C$ near $(\tilde g^{(4)},0)$, then $\hat g^{(4)}\in\mathcal E$ and it satisfies the gauge-free condition $\beta_{\tilde g^{(4)}}\hat g^{(4)}=0$. By the analysis in the proof of Theorem 3.2, any $g^{(4)}\in\mathcal E$ near $\tilde g^{(4)}$ is isometric via a diffeomorphism in $\mathcal D_4$ to a gauge free element $\hat g^{(4)}$. So the projection space $P(\mathcal C)$ is a slice for $\mathcal E$ under the action of $\mathcal D_4$. Therefore, locally the space $\mathcal C$ is a fiber bundle over $\mathcal E/\mathcal D_4$, with the fiber at $[g^{(4)}]$ being $\mathcal H_{g^{(4)}}$. Thus near the element $(\tilde g^{(4)},0)$, we have,
$$\mathcal C\cong \mathcal E/\mathcal D_4\times\mathcal H_{\tilde g^{(4)}}.$$
It is proved in [Az] that the moduli space $\mathcal E/\mathcal D_4$ is a smooth Banach manifold, and hence it follows that $\mathcal C$ admits smooth Banach manifold structure.
\end{proof}
It then follows directly from Theorem 3.2 that the domain space $\mathbb E$ of the Bartnik boundary map $\Pi$ is a smooth Banach manifold. In the following sections, we will show that the boundary map $\tilde\Pi$ is Fredholm, which then implies so is $\Pi$.
\section{Ellipticity of BVP (3.2)}~~~

In this section, we will prove the ellipticity of BVP (3.2), implementing the criterion developed by Agmon-Douglis-Nirenberg (cf.[ADN]). We use the following standard notation. Let $\xi$ denote a $1-$form on $M$, $\eta$ denote a nonzero $1-$form tangential to the boundary $\partial M$, and $\mu$ a unit $1-$form normal to the boundary $\partial M$. The index $0$ denotes the direction along $\partial t$ in $V^{(4)}$, and index $1,2,3$ denote the tangential direction on $M$. When restricted on the boundary, index $1$ denotes the (outward) normal direction to $\partial M\subset M$ and indices $2,3$ denote directions tangent to $\partial M$. We use greek letters when $0$ is included in the indices, and latin letters when there are only tangential components involved.

Based on the system (3.2), we define a differential operator $\mathcal F=(\mathcal L,\mathcal B)$ with interior operator $\mathcal L$, mapping a pair $(g^{(4)},F)$ to the interior equations in (3.2):
\begin{equation*}
\begin{split}
\mathcal L:\mathcal S\times C^{m,\alpha}_{\delta}(M)\rightarrow S^{m-2,\alpha}_{\delta+2}(V^{(4)})\times C^{m-2,\alpha}_{\delta+2}(M)\\
\mathcal L(g^{(4)},F)= (~2(Ric_{g^{(4)}}+\delta^*_{g^{(4)}}\beta_{\tilde g^{(4)}}g^{(4)}),~~\Delta F~~);
\end{split}
\end{equation*}
and a boundary operator $\mathcal B$ mapping $(g^{(4)}, F)$ to the boundary equations in $(3.2)$:
\begin{equation*}
\begin{split}
\mathcal B: \mathcal S\times C^{m,\alpha}_{\delta}(M)&\rightarrow \mathbb B\\
\mathcal B(g^{(4)},F)= (~~~
&g_{\partial M},\\
&aH_{\partial M}+btr_{\partial M}K,\\
&atr_{\partial M}K+bH_{\partial M},\\
&\omega_{\mathbf n}+a^2d_{\partial M}(b/a),\\
&\beta_{\tilde g^{(4)}}g^{(4)}~~~).
\end{split}
\end{equation*}
In the above, $S^{m-2,\alpha}_{\delta+2}(V^{(4)})$ denotes the space of symmetric 2-tensors in $V^{(4)}$, which are time independent, $C^{m-2,\alpha}$ smooth and asymptotically zero at the rate of $(\delta+2)$; $\mathbb B$ is an abbreviation of the target space of $\mathcal B$, given by,
$$\mathbb B=S^{m,\alpha}(\partial M)\times [C^{m-1,\alpha}(\partial M)]^2\times \wedge_1^{m-1,\alpha}(\partial M)\times C^{m-1,\alpha}(\partial M)\times \wedge_1^{m-1,\alpha}(\partial M).$$  We refer to \S7.1 for these notations of tensor spaces. 

\begin{theorem}
The linearization $D\mathcal F$ of $\mathcal F$ at $(\tilde g^{(4)},0)$ is elliptic.
\end{theorem}
\begin{proof}
We use the characterization of ellipticity in [ADN] to prove the thoerem. We first show in \S4.1 that $D\mathcal F$ is properly elliptic. Then in \S4.2 we show that $D\mathcal F$ satisfies the complementing boundary condition.
\medskip
\subsection{Properly elliptic condition}~~~

The linearization of the interior operator at $(\tilde g^{(4)},0)$ is given by (cf.[Be])
\begin{equation*}
\begin{split}
D\mathcal L: T_{\tilde g^{(4)}}\mathcal S\times C^{m,\alpha}_{\delta}(M)&\rightarrow S^{m-2,\alpha}_{\delta+2}(V^{(4)})\times C^{m-2,\alpha}_{\delta+2}(M)\\
D\mathcal L(h^{(4)},G)&= (~D_{\tilde g^{(4)}}^*D_{\tilde g^{(4)}} h^{(4)},~\Delta G~).
\end{split}
\end{equation*}
Let $\{e_i\},~(i=1,2,3)$ be a local orthonormal basis of the tangent bundle on $M$. Recall that $\mathbf N$ denotes the future pointing time-like unit vector perpendicular to $M$ in the spacetime. Based on (2.1), $\mathbf N=N^{-1}(\partial_t-X)$, with $X,N$ being the shift vector and lapse function of $\tilde g^{(4)}$. Then the Laplacian $D_{\tilde g^{(4)}}^*D_{\tilde g^{(4)}} h^{(4)}_{\alpha\beta}$ in the above can be expressed in the $3+1$ slice formalism (2.1) of the spacetime as:
\begin{equation*}
\begin{split}
D_{\tilde g^{(4)}}^*D_{\tilde g^{(4)}} h^{(4)}_{\alpha\beta}&=-D_{\mathbf N}D_{\mathbf N} h^{(4)}_{\alpha\beta}+\Sigma_{i=1}^3D_{e_i}D_{e_i} h^{(4)}_{\alpha\beta}+O_1(h^{(4)})\\
 &=-D_{\frac{1}{N}(\partial t-X)}D_{\frac{1}{N}(\partial t-X)} h^{(4)}_{\alpha\beta}+\Sigma_{i=1}^3D_{e_i}D_{e_i} h^{(4)}_{\alpha\beta}+O_1(h^{(4)})\\
 &=-\frac{1}{N^2}\partial_X\partial_X h^{(4)}_{\alpha\beta}+\Sigma_{i=1}^3\partial_{e_i}\partial_{e_i} h^{(4)}_{\alpha\beta}+O_1(h^{(4)}).\end{split}
\end{equation*}
Here $O_1(h^{(4)})$ denotes those terms with lower($\leq 1$) order derivatives.
A similar formula holds for the term $\Delta G$, i.e.
\begin{equation*}
\begin{split}
\Delta G =-\frac{1}{N^2}\partial_X\partial_X G+\Sigma_{i=1}^3\partial_{e_i}\partial_{e_i} G+O_1(G).\end{split}
\end{equation*}

Thus, the matrix of principal symbol for $D\mathcal L$ is given by,
\begin{equation}
L(\xi)=a(\xi)I_{11\times 11}\end{equation}
with
 \begin{equation}
 a(\xi)=\xi_1^2+\xi_2^2+\xi_3^2-\frac{1}{N^2}(X_i\xi^i)^2.
 \end{equation} 
The determinant of this matrix is obviously
$$\text{det}(L(\xi))=[a(\xi)]^{11}.$$
Notice that $\frac{||X||^2}{N^2}<1$ by $(2.2)$ and hence,
$$a(\xi)=|\xi|^2-\langle\frac{X}{N},\xi\rangle^2\geq |\xi|^2-\frac{||X||^2}{N^2}|\xi|^2>0,$$
Therefore, the interior operator $L$ is properly elliptic. 
\\
\subsection{Complementing boundary condition}~~~

The complementing boundary condition is defined as (cf.[ADN]): 
\medskip

\emph{Let $L^*(\xi)$ be the adjoint matrix of $L(\xi)$ and set $\xi=\eta+z\mu$. The rows of the matrix $[B\cdot L^*](\eta+z\mu)$ are linearly independent modulo $l^+(z)=\prod (z-z_k)$, where $\{z_k\}$ are the roots of $detL(\eta+z\mu)=0$ having positive imaginary parts.}
\medskip
\\Since the principal symbol of $L$ is the identity matrix (up to a scalar) as shown in $(4.1)$, the complementing condition will hold as long as the boundary matrix $B(\eta+z\mu)$ is non-degenerate when $z$ is a root of $\text{det}L(\eta+z\mu)=0$ with positive imaginary part.

The linearization of the boundary operator $\mathcal B$ at $(\tilde g^{(4)},0)$ is given by,
\begin{equation}
\begin{split}
\mathcal B:T\mathcal S\times C^{m,\alpha}_{\delta}(M)&\rightarrow \mathbb B\\
D\mathcal B(h^{(4)},G)= (~&h_{\partial M}\\
&(H_{\partial M})'_{h^{(4)}}+O_0(G)\\
&tr_{\partial M}K'_{h^{(4)}}+O_0(G)\\
&(\omega_{\mathbf n})'_{h^{(4)}}+Nd_{\partial M}G+O_0(G)\\
&\beta_{ g^{(4)}}h^{(4)}~).
\end{split}
\end{equation}
Here we use the notation $T'_{h^{(4)}}$ to denote the variation of the tensor $T$ with respect to the deformation $h^{(4)}$. Notice that at $(\tilde g^{(4)},0)$, $a=1,~b=0$. The formula $(3.3)$ of the scalar field $a$ involves only the $0-$order information of  $F$. Thus the 2nd and 3rd lines in the expression of $D\mathcal B$ above, which represent the linearization of Bartnik data $(aH_{\partial M}+btr_{\partial M}K)$ and $(atr_{\partial M}K+bH_{\partial M})$ at $(a=1,b=0)$, do not contain high order ($\geq 1$) derivatives of $G$. It is easy to check at $(a=1,b=0)$, the variation $(a^2d_{\partial M}b/a\big)'_G=Nd_{\partial M}G+O_0(G),$
which contributes to the fourth line in $D\mathcal B$.

Based on $(4.3)$, the principal symbol of $\mathcal B$ is of the form:
\begin{equation}
B(\xi)=
\begin{bmatrix}
0_{3\times 8}&
\resizebox{.08\textwidth}{!}{$
\begin{matrix}
1&0&0\\
0&1&0\\
0&0&1
\end{matrix}
$}
\\
\tilde B_{8\times8}&*
\end{bmatrix}.
\end{equation}
Notice that $B(\xi)$ is a $11\times11$ matrix, since the boundary terms in (4.3) contain 11 equations in total and 11 (ordered) unknowns. Here for a simple expression of the boundary matrix, the unknown vector $(h^{(4)},G)$ is particularly arranged in the order of:
$$(G,h^{(4)}_{00},h^{(4)}_{01},h^{(4)}_{02},h^{(4)}_{03},h^{(4)}_{11},h^{(4)}_{12},h^{(4)}_{13},h^{(4)}_{22},h^{(4)}_{23},h^{(4)}_{33}).$$
Obviously, the first boundary term $h_{\partial M}=h^{(4)}_{ij},~(2\leq i\leq j\leq3)$. Thus the first three rows of $B$ in (4.4) contain only zeros in the first eight columns and a $3\times 3$ identity matrix at the end. The remaining eight rows of $B$ represent the symbol of 2nd-5th boundary terms in (4.3), with $\tilde B_{8\times 8}$ denoting the first eight columns which are determined by the $G$ and $h^{(4)}_{\alpha\beta}~~(0\leq\alpha\leq1,~\alpha\leq\beta\leq 3)$ components of the corresponding boundary terms. Detailed calculation given in \S 7.3 shows that the matrix $\tilde B_{8\times 8}$ is given by
$$\tilde B_{8\times 8}=-\frac{1}{2N^2}[(\hat B_1)_{8\times 1}~ (\hat B_2)_{8\times 4}~(\hat B_3)_{8\times 3}].$$
Here the scalar $-\frac{1}{2N^2}$ is for the purpose that the matrix following it can be written in a simpler way. Since $G$ only appears as $Nd_{\partial M}G$ in the fourth boundary term in (4.3), it is easy to derive that the first column of $\tilde B$ is given by
$$
\hat B_1=
\begin{bmatrix}
0&0&-2N^3\xi_2&-2N^3\xi_3&0&0&0&0
\end{bmatrix}^t
$$
In the above, we have the extra factor $-2N^2$, because $-\frac{1}{2N^2}$ has been factored out from these rows, which contributes to the factor in the front of the expression of $\tilde B$.
Based on the symbol calculation in \S7.3, we obtain the 2nd-5th columns of $\tilde B$ as, 
$$
\hat B_2=\resizebox{.66\textwidth}{!}{$
\begin{bmatrix}
0&0&0&0\\
0&0&2N\xi_2&2N\xi_3\\
0&N\xi_2&N\xi_1&0\\
0&N\xi_3&0&N\xi_1\\
\xi_1&2S-2\xi_1X^1&-2\xi_1X^2&-2\xi_1X^3~\\
\xi_2&-2\xi_2X^1&2S-2\xi_2X^2&-2\xi_2X^3~\\
\xi_3&-2\xi_3X^1&-2\xi_3X^2&2S-2\xi_3X^3~\\
2S&2(N^2\xi_1-SX^1)&2(N^2\xi_2-SX^2)&2(N^2\xi_3-SX^3)
\end{bmatrix}$},
$$
and the last three columns are given by
$$
\hat B_2=\resizebox{.9\textwidth}{!}{$
\begin{bmatrix}
0&2N^2\xi_2&2N^2\xi_3\\
0&-2N\xi_2X^1&-2N\xi_3X^1\\
-N\xi_2X^1&N\xi_3X^3&-N\xi_2X^3\\
-N\xi_3X^1&-N\xi_3X^2&N\xi_2X^2\\
\xi_1X^1X^1+N^2\xi_1-2S X^1&\xi_1X^1X^2-2S X^2+2N^2\xi_2&\xi_1X^1X^3-2S X^3+2N^2\xi_3\\
\xi_2X^1X^1-N^2\xi_2&\xi_2X^1X^2-2S X^1+2N^2\xi_1&\xi_2X^1X^3\\
\xi_3X^1X^1-N^2\xi_3&\xi_3X^1X^2&\xi_3X^1X^3-2S X^1+2N^2\xi_1\\
0&0&0
\end{bmatrix}$},
$$
inside which $S=\xi_1X^1+\xi_2X^2+\xi_3X^3$. 

Obviously, to prove the complementing boundary condition, it suffices to verify that $\tilde B(\eta+z\mu)$ is nonsingular when $z$ is a root of $a(\eta+z\mu)$ in (4.2) with positive imaginary part.
We can simplify $\tilde B$ using elementary row and column operation of matrices (cf.$\S 7.4$ for the detailed calculations) and obtain an equivalent matrix $\hat B$ so that $\text{det}\tilde B=-\frac{1}{32N^{11}}\text{det}\hat B$. The matrix $\hat B(\xi)$ is given by,
\begin{equation}
\resizebox{.90\textwidth}{!}{$
\begin{bmatrix}
0&0&0&0&0&0&-\xi_2&-\xi_3\\
0&0&0&\xi_2&\xi_3&0&0&0\\
-2N^2\xi_2&0&0&\xi_1&0&0&\xi_1X^1+\xi_3X^3&-\xi_2X^3\\
-2N^2\xi_3&0&0&0&\xi_1&0&-\xi_3X^2&\xi_1X^1+\xi_2X^2\\
0&\xi_1&2S&0&0&2N^2\xi_1&-2S X^2&-2S X^3\\
0&\xi_2&0&2S&0&0&2N^2\xi_1&0\\
0&\xi_3&0&0&2S&0&0&2N^2\xi_1\\
0&S&N^2\xi_1+SX^1&SX^2&SX^3&N^2S+N^2\xi_1X^1&0&0
\end{bmatrix}$}.
\end{equation}
Computing the determinant of the matrix above gives
$$\text{det}(\hat B)(\xi)=8N^8(\xi_1^2-\frac{S^2}{N^2})^2(\xi_1^2+\xi_2^2)^2.$$
If $\xi=\eta+z\mu$, then 
$$\text{det}(\hat B)(\eta+z\mu)=8N^8(z^2-\frac{\langle X,\eta+z\mu\rangle^2}{N^2})^2|\eta|^4.$$
If $z$ is a complex root of $a(\eta+z\mu)=0$, then from (4.2) it follows,
$$|\eta+z\mu|^2-\frac{1}{N^2}\langle X,\eta+z\mu\rangle^2=0,$$ 
i.e. $|\eta|^2+z^2=\frac{\langle X,\eta+z\mu\rangle^2}{N^2}$, and thus
\begin{equation*}
\begin{split}
\text{det}(\tilde B)(\eta+z\mu)&=8N^8(z^2-\frac{\langle X,\eta+z\mu\rangle^2}{N^2})^2|\eta|^4\\
&=8N^8(z^2-z^2-|\eta|^2)^2|\eta|^4\\
&=8N^8|\eta|^8,
\end{split}
\end{equation*}
which is obviously nonzero for $\eta\neq 0$. Thus the complementing boundary condition holds. This finishes the proof of Theorem 4.1. 
 \end{proof}
 
 It then follows from Theorem 4.1 that the linearization of BVP (3.2) is elliptic, which further implies that the boundary map $\tilde \Pi$ defined in \S3 is Fredholm -- the linearization $D\tilde \Pi$ is a Fredholm operator at $(\tilde g^{(4)},0)$. Now according to the Theorem 3.2, we can conclude that Theorem 1.1 is true on condition of the Assumption 3.1, i.e.
 \begin{theorem}
 If $\tilde g^{(4)}$ is a stationary vacuum spacetime metric such that the Assumption 3.1 holds, then the moduli space $\mathbb E$ admits smooth Banach manifold structure near $[\tilde g^{(4)}]$, and the boundary map $\Pi$ is Fredholm at $[\tilde g^{(4)}]$. 
 \end{theorem}
 To conclude this section, recall that in \S 3 we show that $\mathcal C$ can be interpreted geometrically as a local coordinate chart of the moduli space $\mathbb E$, and the map $\tilde\Pi$ is exactly the map $\Pi$ expressed in this chart. However, such a local chart is effective only if Assumption 3.1 holds.  In the following section, we will develop an alternative local chart at a reference metric $\tilde g^{(4)}\in\mathcal E$ where Assumption 3.1 fails. Furthermore, we show that the ellipticity result still holds in this case.

\begin{remark}
The operator $\beta_{\tilde g^{(4)}}\delta^*_{\tilde g^{(4)}}$ with Dirichlet boundary condition is elliptic and self-adjoint. This is shown in \S5.1 using the quotient formalism of stationary spacetimes. When expressed on the quotient manifold $(S,g_S)$ (cf.(5.9)), this operator is in the form of the Laplace operator plus lower order terms -- especially nontrivial 0-order terms generated by the twist tensor of the metric. If the spacetime metric $\tilde g^{(4)}$ is static, then the twist tensor is zero and the operator $\beta_{\tilde g^{(4)}}\delta^*_{\tilde g^{(4)}}$ reduces to the Laplacian, which is invertible. However, if the metric is not static, the 0-order terms in the operator are not necessarily vanishing or positive, so they may result in a nontrivial kernel of the operator, in which case Assumption 3.1 might fail. On the other hand, because of ellipticity and self-adjointness, this operator must be invertible at least for generic metrics in the space $\mathcal E$. It would be interesting to understand whether invertibility holds for all $g^{(4)}\in\mathcal E$.
\end{remark}
\section{Alternative charts}
In this section, we assume that $\tilde g^{(4)}$ is a fixed stationary vacuum metric where Assumption 3.1 fails.
\subsection{Perturbation of the metric}$~~$

We will use the projection formalism of stationary spacetimes (cf.[Kr],[G]) in this subsection. 
In a globally hyperbolic stationary spacetime $(V^{(4)},g^{(4)})$, the Killing vector field $\partial_t$ generates an isometric and proper $\mathbb R-$action on the spacetime. Let $S$ be the orbit space of this action, i.e. $S=V^{(4)}/\mathbb R$. Then $S$ is a smooth 3-manifold and inherits a Riemannian metric $g_S$, which is the restriction of the metric $g^{(4)}$ to the horizontal distribution --- the orthogonal complement of $\text{span}\{\partial_t\}$ in $TV^{(4)}$. It is shown in [G] that there is a one-to-one correspondence between tensor fields $T^{'b...d}_{a...c}$ on $S$ and tensor fields $T^{b...d}_{a...c}$ on $V^{(4)}$ which satisfy
\begin{equation*}
(\partial_t)^aT^{b...d}_{a...c}=0,...,(\partial_t)_dT^{b...d}_{a...c}=0\text{ and }L_{\partial_t}T^{b...d}_{a...c}=0.
\end{equation*}
In the following we will identify tensor fields being on $S$ as tensor fields in $V^{(4)}$ satisfying the conditions above. For example, a vector field $X'\in TS$ corresponds to a vector field $X$ in $V^{(4)}$ such that $\langle X,\partial_t\rangle_{g^{(4)}}=0$ and $L_{\partial_t}X=0$. So we will drop the prime -- identify $X'$ as $X$.

In the projection formalism, any stationary spacetime metric $g^{(4)}$ is globally of the form
\begin{equation*}
g^{(4)}=-u^2(dt+\theta)^2+g_S.
\end{equation*}
Here $g_S$ is the metric on $S$, which can also be interpreted as a time-independent tensor field in $V^{(4)}$ as discussed above, and $\theta$ is a 1-form on $S$ (in the same sense) such that $-u^2(dt+\theta)$ is the dual of $\partial_t$ with respect to $g^{(4)}$.
\begin{remark}
In our case $V^{(4)}\cong\mathbb R\times M$, there is a natural diffeomorphism between the quotient manifold $S$ and the hypersurface $M=\{t=0\}$. Thus we can pull back the radius function on $M$ to $S$ and define weighted H\"older spaces of tensor fields on $S$ similarly as in \S7.1. So if a vector field $Y\in T^{m,\alpha}_{\delta}(V^{(4)})$ satisfies $\langle Y,\partial_t\rangle_{g^{(4)}}=0$ and $L_{\partial_t}Y=0$, then it can be identified with a tensor field $Y\in T^{m,\alpha}_{\delta}(S)$
\end{remark}

Suppose that in the projection formalism, $\tilde g^{(4)}$ is expressed as,
\begin{equation}
\tilde g^{(4)}=-u^2(dt+\theta)^2+g_S.
\end{equation}
Take a smooth curve (parametrized by $\epsilon$) of perturbations of $\tilde g^{(4)}$ given by,
\begin{equation}g^{(4)}_{\epsilon}=\tilde g^{(4)}+\epsilon(dt+\theta)^2.
\end{equation}
First we prove the following property of this family of metrics.
\begin{proposition}
The metric $g_{\epsilon}^{(4)}$ is Bianchi-free with respect to $\tilde g^{(4)}$, i.e.
\begin{equation*}
\beta_{\tilde g^{(4)}}g^{(4)}_{\epsilon}=0.
\end{equation*}
\end{proposition}
\begin{proof}
Clearly by (5.2),
$$\beta_{\tilde g^{(4)}}g_{\epsilon}^{(4)}=\epsilon\beta_{\tilde g^{(4)}}(dt+\theta)^2.$$
Let
\begin{equation}\alpha=(dt+\theta),\end{equation}
then $\alpha=-\frac{1}{u^2}\xi$, where $\xi=-u^2(dt+\theta)$ is the dual of $\partial_t$. Obviously $\alpha(\partial_t)=1$, $\alpha(v)=0,~\forall v\in TS$, and hence $tr_{\tilde g^{(4)}}\alpha^2=-u^{-2}$. As a result, \begin{equation}
\begin{split}
\beta_{\tilde g^{(4)}}(\alpha^2)&=\delta_{\tilde g^{(4)}}(\alpha^2)+\frac{1}{2}d(tr_{\tilde g^{(4)}}\alpha^2)=\delta_{\tilde g^{(4)}}(\alpha^2)+u^{-3}du.
\end{split}
\end{equation}
For the divergence term above, we have
\begin{equation}
\begin{split}
\delta_{\tilde g^{(4)}}(\alpha^2)&=-\frac{1}{u^2}\{-\nabla_{\partial_t}[\alpha^2(\partial_t)]+\alpha^2(\nabla_{\partial_t}\partial_t)\}\\
&=\frac{1}{u^2}\nabla_{\partial_t}\alpha=-\frac{1}{u^2}\nabla_{\partial_t}(\frac{1}{u^2}\xi)\\
&=-\frac{1}{u^4}\nabla_{\partial_t}\xi=-u^{-3}d u.
\end{split}
\end{equation}
In the first equality above, we have used the fact that $\nabla_{\partial_t}\partial_t=u\nabla u$ (cf.\S7.5), which gives a vector field on $S$ and so $\alpha(\nabla_{\partial_t}\partial_t)=0$. In the last equality, trivially we have $\nabla_{\partial_t}\xi=udu$.  Equations (5.4) and (5.5) now imply that $\alpha^2$ is Bianchi-free.
\end{proof}
In addition to Bianchi-free, the family of metrics $g_{\epsilon}^{(4)}$ possesses another property --- for generic $\epsilon$ the operator $\beta_{\tilde g^{(4)}}\delta^*_{g_{\epsilon}^{(4)}}$ is invertible, in the following sense:   
\begin{proposition} In any neighborhood $I$ of $0$, there is an $\epsilon\in I$ such that the BVP with unknown $Y\in T^{m,\alpha}_{\delta}(V^{(4)})$ given by,
\begin{equation}
	\begin{cases}
		\beta_{\tilde g^{(4)}}\delta^*_{g_{\epsilon}^{(4)}}Y=0\quad\text{on }S\\
		Y=0\quad\text{on }\partial S
	\end{cases}
\end{equation}
has only the trivial solution $Y=0$.
\end{proposition}
To prove this proposition, we state the following lemma first. 
\begin{lemma}
 The BVP (5.6) is elliptic (for $\epsilon$ small) and formally self-adjoint. 
\end{lemma}
\begin{proof}
Since $\delta^*_{g_{\epsilon}^{(4)}}Y=\frac{1}{2}L_{Y}g_{\epsilon}^{(4)}=\frac{1}{2}L_{Y}(\tilde g^{(4)}+\epsilon\alpha^2)=\delta^*_{\tilde g^{(4)}}Y+\frac{\epsilon}{2}L_Y\alpha^2$, one has,
\begin{equation}
\beta_{\tilde g^{(4)}}\delta^*_{g_{\epsilon}^{(4)}}Y=\beta_{\tilde g^{(4)}}\delta^*_{\tilde g^{(4)}}Y+\frac{\epsilon}{2}\beta_{\tilde g^{(4)}}L_Y\alpha^2,
\end{equation}
where $\alpha$ is as defined in (5.3). 

Notice that any time-independent vector field $Y$ in $V^{(4)}$ can be decomposed into a vector field on $S$ and another part proportional to $\partial_t$. Let $Y^{\perp}=u^{-1}\langle Y,\partial_t\rangle$ and $Y^T=Y+u^{-1}Y^{\perp}\partial_t$. Obviously, $Y^{\perp}$ is independent of $t$, so it can be taken as a function on $S$. It is also easy to verify that $\langle Y^T,\partial_t\rangle_{g^{(4)}}=0$ and $L_{\partial_t}Y^T=0$, i.e. $Y^T$ is a vector fields on $S$. Thus we have the decomposition 
\begin{equation}
Y=Y^T-u^{-1}Y^{\perp}\partial_t.
\end{equation}
We decompose the vector $\beta_{\tilde g^{(4)}}\delta^*_{\tilde g^{(4)}}Y$ on the right side of equation (5.7) in the same way as described above. It is shown in \S7.5 that the components are given by,
\begin{equation}
\begin{cases}
[\beta_{\tilde g^{(4)}}\delta^*_{\tilde g^{(4)}}Y]^T=(\nabla_{g_S})^*\nabla_{g_S}Y^T+u^{-2}Y^T(u)\nabla u
-u^{-1}(\nabla_{g_S})_{\nabla u}Y^T\\
\quad\quad\quad\quad\quad\quad\quad\quad\quad\quad\quad\quad\quad\quad+\frac{1}{2}u^2d\theta(d\theta(Y^T))-u^2d\theta(\nabla\frac{Y^{\perp}}{u})\\
~~~\\
-[\beta_{\tilde g^{(4)}}\delta^*_{\tilde g^{(4)}}Y]^{\perp}=-u\Delta_{g_S}(\frac{Y^{\perp}}{u})+3\langle\nabla\frac{Y^{\perp}}{u},\nabla u\rangle-\frac{1}{4}u^2Y^{\perp}|d\theta|^2
\\
\quad\quad\quad\quad\quad\quad\quad\quad\quad\quad\quad\quad\quad\quad\quad+u\langle d\theta,\nabla_{g_S}Y^T\rangle-d\theta(\nabla u,Y^T),
\end{cases}
\end{equation}
where $\nabla_{g_S}$ (and $\Delta_{g_S}$) denotes connection (and Laplace operator) of $g_S$.  Notice the leading terms of the operator in (5.9) are the Laplacian $[(\nabla_{g_S})^*\nabla_{g_S}Y^T]$ and $[-u\Delta_{g_S}(\frac{Y^{\perp}}{u})]$. Thus $\beta_{\tilde g^{(4)}}\delta^*_{\tilde g^{(4)}}$ is an elliptic operator with respect to the Dirichlet boundary condition, and so is the operator $\beta_{\tilde g^{(4)}}\delta^*_{g_{\epsilon}^{(4)}}$ (for $\epsilon$ small) according to the expression (5.7).

Let $Y_1,Y_2\in T^{m,\alpha}_{\delta}(V^{(4)})$ be two asymptotically zero time-independent vector fields which are vanishing along $\partial V^{(4)}$. Then, 
\begin{equation*}
\begin{split}
&\int_{S}\langle\beta_{\tilde g^{(4)}}\delta^*_{\tilde g^{(4)}}Y_1,Y_2\rangle_{g^{(4)}}\cdot u\cdot dvol_{g_S}\\
=&
\int_{S}\{\langle[\beta_{\tilde g^{(4)}}\delta^*_{\tilde g^{(4)}}Y_1]^T,Y_2^T\rangle_{g_S}+(-[\beta_{\tilde g^{(4)}}\delta^*_{\tilde g^{(4)}}Y_1]^{\perp})\cdot Y_2^{\perp}\}\cdot u\cdot dvol_{g_S}
\end{split}
\end{equation*} 
Substituting equations in (5.9) into the integral above and then integrating by parts gives,
\begin{equation}
\begin{split}
&\int_{S}\{\langle[\beta_{\tilde g^{(4)}}\delta^*_{\tilde g^{(4)}}Y_1]^T,Y_2^T\rangle_{g_S}+(-[\beta_{\tilde g^{(4)}}\delta^*_{\tilde g^{(4)}}Y_1]^{\perp})\cdot Y_2^{\perp}\}\cdot u\cdot dvol_{g_S}\\
&=\int_{S}\{\langle[\beta_{\tilde g^{(4)}}\delta^*_{\tilde g^{(4)}}Y_2]^T,Y_1^T\rangle_{g_S}+(-[\beta_{\tilde g^{(4)}}\delta^*_{\tilde g^{(4)}}Y_2]^{\perp})\cdot Y_1^{\perp}\}\cdot u\cdot dvol_{g_S}\\
&\quad+(\int_{\partial S}+\int_{\infty})[B(Y_2,Y_1)-B(Y_1,Y_2)],
\end{split}
\end{equation}
where $B(Y_2,Y_1)=u\langle\nabla_{\mathbf n}Y_2^T,Y_1^T\rangle]+2u^2d\theta(\mathbf n,Y_1^T)Y_2^{\perp}+u\mathbf n(Y_1^{\perp})Y_2^{\perp}$. It is obvious that the boundary integral on $\partial S$ is zero, since $Y_1,Y_2$ vanish on the boundary. The boundary term at infinity $\int_{\infty}=\lim_{r\to\infty}\int_{S_r}$, with $S_r$ denoting the sphere of radius $r$ on $S$. It is also zero because the decay rate of the bilinear form $B(Y_1,Y_2)$ is $2\delta+1>2$. (cf. Remark 5.1) Thus it follows that, the differential operator (5.9) is formally self-adjoint with respect to the measure $u\cdot dvol_{g_S}$ on $S$.

\medskip
$Remark.$ One has the following integration by parts formula in the spacetime $(V^{(4)},g^{(4)})$:
\begin{equation*}
\begin{split}
\int_{V^{(4)}}\langle \nabla_{\tilde g^{(4)}}^*\nabla_{\tilde g_{(4)}}Y_1,Y_2\rangle_{\tilde g^{(4)}}&dvol_{\tilde g^{(4)}}=
\int_{V^{(4)}}\langle \nabla_{\tilde g^{(4)}}^*\nabla_{\tilde g_{(4)}}Y_2,Y_1\rangle_{\tilde g^{(4)}}dvol_{\tilde g^{(4)}}\\
&+\int_{\partial V^{(4)}}\langle (\nabla_{\tilde g^{(4)}})_{\mathbf n}Y_2,Y_1\rangle_{\tilde g^{(4)}}-\langle (\nabla_{\tilde g^{(4)}})_{\mathbf n}Y_1,Y_2\rangle_{\tilde g^{(4)}}.
\end{split}
\end{equation*}
When the spacetime $(V^{(4)},\tilde g^{(4)})$ is stationary, the equation above reduces to the equation (5.10) on the quotient manifold $(S,g_S)$. 
\medskip

Using the same method as above, it is easy to check the following equality holds for any time-independent symmetric 2-tensor $h\in S^{m,\alpha}_{\delta+1}(V^{(4)})$ and vector field $Y\in T^{m,\alpha}_{\delta}(V^{(4)})$ with $Y|_{\partial V^{(4)}}=0$:
\begin{equation}
\begin{split}
\int_{S}\langle\beta_{\tilde g^{(4)}}h,Y\rangle_{\tilde g^{(4)}}\cdot u\cdot dvol_{g_S}=
\int_{S}\langle h,\delta_{\tilde g^{(4)}}^*Y+\frac{1}{2}(\delta_{\tilde g^{(4)}}Y)\tilde g^{(4)}\rangle_{\tilde g^{(4)}} u\cdot dvol_{g_S}.
\end{split}
\end{equation}
Thus, as for the second term on the right side of equation (5.7), we have the following equality for all vector fields $Y_1,Y_2\in T^{m,\alpha}_{\delta}(V^{(4)})$ which vanish at $\partial V^{(4)}$:
\begin{equation*}
\begin{split}
&\int_{S}\langle \beta_{\tilde g^{(4)}}L_{Y_1}\alpha^2,Y_2\rangle_{\tilde g^{(4)}} u\cdot d vol_{g_S}\\
&=\int_{S}\langle L_{Y_1}\alpha^2,\delta^*_{\tilde g^{(4)}}Y_2+\frac{1}{2}(\delta_{\tilde g^{(4)}}Y_2)\tilde g^{(4)}\rangle_{\tilde g^{(4)}} u\cdot d vol_{g_S}\\
&=\int_S-2u^{-2}\langle L_{Y_1}\alpha^2(\partial_t)^T,\delta^*_{\tilde g^{(4)}}Y_2(\partial_t)^T\rangle_{g_S} u\cdot dvol_{g_S}\\
&=\int_{S}-2u^{-2}\langle~ d\theta(Y_1^T)-d(\frac{Y_1^{\perp}}{u}),~-\frac{u^2}{2}[d\theta(Y^T_2)-d(\frac{Y_2^{\perp}}{u})]~\rangle_{g_S}u\cdot d vol_{g_S}\\
&=\int_{S}-2u^{-2}\langle~ -\frac{u^2}{2}[d\theta(Y_1^T)-d(\frac{Y_1^{\perp}}{u})],~d\theta(Y^T_2)-d(\frac{Y_2^{\perp}}{u})~\rangle_{g_S}u\cdot d vol_{g_S}\\
\end{split}
\end{equation*}
\begin{equation*}
\begin{split}
&=\int_S-2u^2\langle \delta^*_{\tilde g^{(4)}}Y_1(\partial_t)^T,L_{Y_2}\alpha^2(\partial_t)^T\rangle_{g_S}u\cdot dvol_{g_S}\\
&=\int_{S}\langle \delta^*Y_1+\frac{1}{2}(\delta_{\tilde g^{(4)}}Y_1)\tilde g^{(4)}, L_{Y_2}\alpha^2\rangle_{\tilde g^{(4)}} u\cdot d vol_{g_S} \\
&=\int_{S}\langle Y_1,  \beta_{\tilde g^{(4)}}L_{Y_2}(dt+\theta)^2\rangle_{\tilde g^{(4)}} u\cdot d vol_{g_S}.
\end{split}
\end{equation*}

In the calculation above, the first equality comes from formula (5.11). The second and third equalities are based on the decompositions (equations (5.12-13) below) of the time-independent vector fields $L_{Y_1}\alpha^2$ and $\delta^*_{\tilde g^{(4)}}Y_2$. Furthermore, the last three equalities above are carried out in the reversed way as the first three. 

The time-independent vector fields $L_{Y_1}\alpha^2$ and $\delta^*_{\tilde g^{(4)}}Y_2$ can be decomposed as,
\begin{equation}
\begin{split}
	\begin{cases}
	        [L_{Y_1}\alpha^2]^T=0\\
		[L_{Y_1}\alpha^2](\partial t,\partial t)=0\\
		\{[L_{Y_1}\alpha^2](\partial t)\}^T=d\theta( Y_1^T)-d(\frac{ Y_1^{\perp}}{u}),
	\end{cases}
\end{split}
\end{equation}
and
\begin{equation}
\begin{split}
\begin{cases}
		(\delta^*_{\tilde g^{(4)}}Y_2)^T=\delta^*_{g_S} Y_2^T\\
		\delta^*_{\tilde g^{(4)}}Y_2(\partial_t,\partial_t)=-u Y_2^T(u)\\
		[\delta^*_{\tilde g^{(4)}}Y_2(\partial_t)]^T=-\frac{1}{2}u^2d\theta( Y_2^T)+\frac{1}{2}u^2d(\frac{ Y_2^{\perp}}{u}).
	\end{cases}
\end{split}
\end{equation}
We refer to \S 7.4 for detailed proof of the equations (5.12-13). 

Summing up all the facts above, we conclude that the system (5.6) is formally self-adjoint.
\end{proof}

Now we give the proof for Proposition 5.3.
\begin{proof}
We prove it by contradiction. Assume that Proposition 5.3 is not true, so there exists an interval $I$ which contains $0$ such that for any $\epsilon\in I$, the system (5.6) has a nonzero solution.

From Lemma 5.4, we see that system $(5.6)$ gives rise to a analytic curve of elliptic self-adjoint operators parametrized by $\epsilon$. By the perturbation theory for self-adjoint operators (cf.[K],[W]), there exists a smooth curve of nontrivial solutions $Y(\epsilon)~(\epsilon\in I)$ solving the system $(5.6)$, i.e.
\begin{equation*}
\begin{cases}
		\beta_{\tilde g^{(4)}}\delta^*_{g_{\epsilon}^{(4)}}Y(\epsilon)=0\quad\text{on }S\\
		Y(\epsilon)=0\quad\text{on }\partial S
	\end{cases}
	\quad\forall \epsilon\in I.
\end{equation*}
The proof of this is discussed in detail in \S7.6. In particular, $Y(0)$ is a nontrivial solution to $(5.6)$ at $\epsilon=0$. In the following we will denote it as $\tilde Y=Y(0)$. Taking the linearization of the system above at $\epsilon=0$, we obtain:
\begin{equation}
	\begin{cases}
		\beta_{\tilde g^{(4)}}\delta_{\tilde g^{(4)}}^* Y'+\beta_{\tilde g^{(4)}}\delta^*_{g'}\tilde Y=0\quad\text{on }S\\
		Y'=0\quad\text{on }\partial S,
	\end{cases}
\end{equation}
where 
$$Y'=\frac{d}{d\epsilon}|_{\epsilon=0}Y(\epsilon),
~\delta^*_{g'}\tilde Y=\frac{d}{d\epsilon}|_{\epsilon=0}\delta^*_{g_{\epsilon}}\tilde Y
=\frac{1}{2}\frac{d}{d\epsilon}|_{\epsilon=0}L_{\tilde Y}g_{\epsilon}=\frac{1}{2}L_{\tilde Y}\alpha^2
.$$

The first equation in (5.14) gives, 
		$$-\beta_{\tilde g^{(4)}}\delta_{\tilde g^{(4)}}^* Y'=\beta_{\tilde g^{(4)}}\delta^*_{g'}\tilde Y.$$
 Since $\beta_{\tilde g^{(4)}}\delta^*_{\tilde g^{(4)}}$ is self-adjoint, the equation above yields that,
\begin{equation}
\begin{split}
	\int_{V^{(4)}}\langle \beta_{\tilde g^{(4)}}\delta^*_{g'}\tilde Y,\tilde Y\rangle dvol_{\tilde g^{(4)}}&=-\int_{V^{(4)}}\langle \beta_{\tilde g^{(4)}}\delta^*_{\tilde g^{(4)}}Y',\tilde Y\rangle dvol_{\tilde g^{(4)}}\\
	&=-\int_{V^{(4)}}\langle Y',\beta_{\tilde g^{(4)}}\delta^*_{\tilde g^{(4)}}\tilde Y\rangle dvol_{\tilde g^{(4)}}\\
	&=0.
\end{split}
\end{equation}
Apply integration by parts to the left side of $(5.15)$ and obtain,
\begin{equation}
	\int_{V^{(4)}}\langle\delta^*_{g'}\tilde Y,\delta^*_{\tilde g^{(4)}}\tilde Y+\frac{1}{2}(\delta \tilde Y)\tilde g^{(4)}\rangle dvol_{\tilde g^{(4)}}=0.
\end{equation}
In the above, $\delta^*_{g'}\tilde Y=\frac{1}{2}L_{\tilde Y}\alpha^2$. Now apply the formulas (5.12-13) to $L_{\tilde Y}\alpha^2$ and $\delta^*_{\tilde g^{(4)}}\tilde Y$, and substitute them into $(5.16)$. It follows that,
$$\int_S\frac{1}{2}||d\theta(\tilde Y^T)-d(\frac{\tilde Y^{\perp}}{u})||_{g_S}^2u\cdot dvol_{g_S}=0.$$
Therefore, we have
\begin{equation}
	d\theta(\tilde Y^T)=d(\frac{\tilde Y^{\perp}}{u}).
\end{equation}
Recall that $\tilde Y$ is a nontrivial solution to system $(5.6)$ at $\epsilon=0$. By applying the decomposition equations in (5.9) to the vector field $\tilde Y$, we express the time-independent system (5.6) (at $\epsilon=0$) as an equivalent system:
\begin{equation}
\begin{cases}
	\nabla_{g_S}^*\nabla_{g_S}\tilde Y^T-\frac{1}{u}(\nabla_{g_S})_{\nabla u}\tilde Y^T+\frac{1}{u^2}\tilde Y^T(u)\nabla u\\
	\quad\quad\quad\quad\quad\quad\quad\quad\quad\quad+\frac{1}{2}u^2d\theta(d\theta(\tilde Y^T))-u^2d\theta(\nabla\frac{\tilde Y^{\perp}}{u})=0\\
	~~\\
	\Delta_{g_S}(\frac{\tilde Y^{\perp}}{u})-3\frac{1}{u}\langle\nabla u,\nabla\frac{\tilde Y^{\perp}}{u}\rangle+\frac{1}{4}uY^{\perp}|d\theta|^2\\
	\quad\quad\quad\quad\quad\quad\quad\quad\quad\quad-\langle d\theta,\nabla_{g_S}\tilde Y^T\rangle+\frac{1}{u}d\theta(\nabla u,Y^T)=0.
\end{cases}
\end{equation} 
Observe that the last two terms in the first equation in $(5.18)$ can be manipulated as:
\begin{equation*}
\begin{split}
	&\frac{1}{2}u^2d\theta(d\theta(\tilde Y^T))-u^2d\theta(\nabla \frac{\tilde Y^{\perp}}{u})\\
	=&\frac{1}{2}u^2d\theta[d\theta(\tilde Y^T)-2d(\frac{\tilde Y^{\perp}}{u})]\\
	=&-\frac{1}{2}u^2d\theta(d\theta(Y^T)),
\end{split}	
\end{equation*}
where the last equality is based on $(5.17)$. Plugging this back to $(5.18)$, we obtain
\begin{equation*}
		\nabla_{g_S}^*\nabla_{g_S}\tilde Y^T-\frac{1}{u}(\nabla_{g_S})_{\nabla u}\tilde Y^T+\frac{1}{u^2}\tilde Y^T(u)\nabla u-\frac{1}{2}u^2d\theta(d\theta(\tilde Y^T))=0
\end{equation*}
Pairing the equation above with $\tilde Y^T$ yields,
\begin{equation*}
		\frac{1}{2}\Delta_{g_S}(||\tilde Y^T||^2)+||\nabla_{g_S}\tilde Y^T||^2-\frac{1}{2u}(\nabla_{g_S})_{\nabla u}||\tilde Y^T||^2+\frac{1}{u^2}||\tilde Y^T(u)||^2+\frac{1}{2}u^2||d\theta(\tilde Y^T)||^2=0
\end{equation*}
Based on this equation and the fact that $\tilde Y^T$ is asymptotically zero and equals to zero on $\partial S$, it is easy to derive by the maximum principle that $\tilde Y^T=0$, and consequently $\tilde Y^{\perp}=0$ according to the second equation in (5.18). This contradicts with the assumption that $\tilde Y$ is nontrivial. 
\end{proof}
Combining Propositions 5.2 and 5.3, it is straightforward to derive that, 

\begin{theorem}In any neighborhood of $\tilde g^{(4)}\in\mathcal S$, there always exists a perturbation $g^{(4)}_0$ of $\tilde g^{(4)}$ such that $\beta_{\tilde g^{(4)}}g^{(4)}_0=0$ (Bianchi-free) and $\beta_{\tilde g^{(4)}}\delta^*_{g^{(4)}_0}$ is invertible.
\end{theorem}
\subsection{Alternative local charts}
\begin{theorem}
Theorem 1.1 still holds without Assumption 3.1.
\end{theorem}
\begin{proof}
In the case Assumption 3.1 fails, we take a perturbation $g^{(4)}_0$ near $\tilde g^{(4)}$ as described in Theorem 5.5. and modify (3.2) to a new BVP with unknowns $( g^{(4)},F)\in \mathcal S\times C^{m,\alpha}_{\delta}(M)$ as follows:
\begin{equation}
	\begin{split}
		&\quad\quad\begin{cases}
			Ric_{g^{(4)}}-\delta^*_{g^{(4)}_0}\beta_{g^{(4)}}g^{(4)}_0=0\\
			\Delta F=0,
		\end{cases}
		\quad\text{on}\quad M\\
		&\quad\quad\begin{cases}			
			g_{\partial M}=\gamma\\
			aH_{\partial M}+btr_{\partial M}K=H\\
			atr_{\partial M}K+bH_{\partial M}=k\\
			\omega_{\mathbf n}+a^2d_{\partial M}(a/b)=\tau\\
			-\beta_{g^{(4)}} g_0^{(4)}=0.\\
		\end{cases}\quad\text{on}\quad\partial M
	\end{split}
\end{equation}

By applying Bianchi operator to the first equation above, one obtains,
\begin{equation}
\begin{cases}
\beta_{g^{(4)}}\delta^*_{g^{(4)}_0}\big(\beta_{g^{(4)}}g^{(4)}_0\big)=0\quad\text{ on }M,\\
-\beta_{g^{(4)}}g^{(4)}_0=0\quad\text{ on }\partial M.
\end{cases}
\end{equation}
Since the operator $\beta_{\tilde g^{(4)}}\delta^*_{g^{(4)}_0}$ is invertible, so is the operator $\beta_{g^{(4)}}\delta^*_{g^{(4)}_0}$ when $g^{(4)}$ is near $\tilde g^{(4)}$. Thus $(5.20)$ implies that $\beta_{g^{(4)}}g^{(4)}_0=0$ on $M$. So a solution $g^{(4)}$ of (5.19) near $\tilde g^{(4)}$ must be Ricci flat and gauge free with respect to $g_0^{(4)}$. As in \S3, to associate the perturbed BVP (5.19) with a natural boundary map, we first construct a solution space $\mathcal C_0$ near $\tilde g^{(4)}$ given by,
$$\mathcal C_0=\{(g^{(4)},F)\in \mathcal S\times C^{m,\alpha}_{\delta}(M):~Ric_{g^{(4)}}=0, \beta_{g^{(4)}}g^{(4)}_0=0,\Delta F=0~\text{on}~M~\}.$$
Obviously, $(\tilde g^{(4)},0)\in\mathcal C_0$ by construction. Next, as in the proof of Theorem 3.2, we need to prove that any stationary vacuum metric $g^{(4)}$ near $\tilde g^{(4)}$ can be transformed by a diffeomorphism in $\mathcal D_4$ so that it satisfies the gauge condition $\beta_{g^{(4)}}g^{(4)}_0=0$. Consider the following map:
\begin{equation*}
\begin{split}
&\mathcal G: \mathcal S\times \mathcal D_4\rightarrow (\wedge_1)^{m.\alpha}_{\delta} (V^{(4)})\\
&\mathcal G(g^{(4)},\Phi)=\beta_{\Phi^*g^{(4)}}g^{(4)}_0.
\end{split}
\end{equation*}
Notice that 
\begin{equation*}
\beta_{\Phi^*g^{(4)}}g^{(4)}_0=\Phi^*\{\beta_{g^{(4)}}[(\Phi^*)^{-1}g^{(4)}_0]\}.
\end{equation*}
Thus the linearization of $\mathcal G$ at $(\tilde g^{(4)}, Id)$ is given by,
\begin{equation*}
\begin{split}
&D\mathcal G|_{(\tilde g^{(4)},Id)}: T\mathcal S\times T\mathcal D_4\rightarrow (\wedge_1)^{m.\alpha}_{\delta} (V^{(4)})\\
&D\mathcal G|_{(\tilde g^{(4)},Id)}[(h^{(4)},Y)]=-\beta_{\tilde g^{(4)}}\delta_{g^{(4)}_0}^*Y+\beta'_{h^{(4)}}g^{(4)}_0. 
\end{split}
\end{equation*}
Since in the linearization above, the operator $[-\beta_{\tilde g^{(4)}}\delta_{g^{(4)}_0}^*]$ is invertible, it follows by the implicit function theorem that, for any $g^{(4)}$ near $\tilde g^{(4)}$, there is a unique element $\Phi\in\mathcal D_4$ near to $Id_{V^{(4)}}$ such that the gauge term $\beta_{\Phi^*g^{(4)}}g^{(4)}_0$ vanishes.

Therefore the perturbed solution space $\mathcal C_0$ has similar structure as the space $\mathcal C$ in \S3, i.e. near $(\tilde g^{(4)},0)$, $\mathcal C_0$ is locally a fiber bundle over the quotient space $\mathcal E/\mathcal D_4$ with fiber being the space of harmonic functions in $C^{m,\alpha}_{\delta}(M)$. Furthermore, based on the Theorems 3.2 and 3.3, we conclude there exists a local diffeomorphism $\mathcal P_0$ such that
$\mathcal C_0\cong\mathbb E$ near $(\tilde g^{(4)},0)$ via $\mathcal P_0$ and 
\begin{equation}
\Pi_0=\Pi\circ\mathcal P_0,
\end{equation}
where $\Pi_0$ is the natural boundary map defined on $\mathcal C_0$ given by,
\begin{equation*}
	\begin{split}
		\Pi_0:\mathcal C_0&\rightarrow\mathbf B\\
		\Pi_0(g^{(4)},F)=(g_{\partial M},a H_{\partial M}+btr_{\partial M} K,&atr_{\partial M} K+b H_{\partial M},\omega_{\mathbf n}+a^2d_{\partial M}(b/a)).
	\end{split}
\end{equation*}
 
As for ellipticity of the system (5.19), notice that linearization of the equality $\beta_{g^{(4)}}g^{(4)}=0$ yields, 
$$(\beta_{g^{(4)}})'_{h^{(4)}}g^{(4)}=-\beta_{g^{(4)}}{h^{(4)}}.$$ 
Thus the linearization of the gauge term in (5.19) at $(\tilde g^{(4)},0)$ is given by:
\begin{equation*}
\begin{split}
[-\delta^*_{g^{(4)}_0}\beta_{g^{(4)}}g^{(4)}_0]'_{h^{(4)}}&=-\delta^*_{g^{(4)}_0}(\beta_{\tilde g^{(4)}})'_{h^{(4)}}g^{(4)}_0\\
&=-\delta^*_{g^{(4)}_0}(\beta_{\tilde g^{(4)}})'_{h^{(4)}}(\tilde g^{(4)}+g^{(4)}_0-\tilde g^{(4)})\\
&=\delta^*_{g^{(4)}_0}\beta_{\tilde g^{(4)}}{h^{(4)}}-\delta^*_{g^{(4)}_0}(\beta_{\tilde g^{(4)}})'_{h^{(4)}}(g^{(4)}_0-\tilde g^{(4)})
\end{split}
\end{equation*}
Comparing the system (5.19) with the previous one (3.2), it is easy to see that, at the reference metric $\tilde g^{(4)}$, the only differences between their linearizations  are given by the term
\begin{equation}
[\delta^*_{g^{(4)}_0}\beta_{\tilde g^{(4)}}-\delta^*_{\tilde g^{(4)}}\beta_{\tilde g^{(4)}}]({h^{(4)}})-\delta^*_{g^{(4)}_0}(\beta_{\tilde g^{(4)}})'_{h^{(4)}}(g^{(4)}_0-\tilde g^{(4)})
\end{equation}
in the interior equations, and
\begin{equation}
-(\beta)'_{h^{(4)}}(g_0^{(4)}-\tilde g^{(4)})
\end{equation}
in the boundary equations.  We can choose $g^{(4)}_0$ close enough to $\tilde g^{(4)}$ so that the terms in (5.22-23) are very small. Then the principal symbols of the perturbed system $(5.19)$ is close to that of the system (3.2). It has been proved that $(3.2)$ is elliptic. So (5.19) is also elliptic. 
As a consequence $\Pi_0$ is a Fredholm map and hence so is $\Pi$ because of the equivalence relation (5.21). This completes the proof.

\end{proof}
\section{Local existence and uniqueness}
In this section we choose $V^{(4)}=\mathbb R\times(\mathbb R^3\setminus B^3)\subset \mathbb R^4$, equipped with the standard coordinates $\{t,x^i\}~(i=1,2,3)$ induced from $\mathbb R^4$. Let $M$ be the hypersurface $\{t=0\}$. Then $\partial M=S^2$, the unit sphere. Set the reference metric $\tilde g^{(4)}=\tilde g^{(4)}_0$, where $\tilde g^{(4)}_0$ is the standard flat (Minkowski) metric on $\mathbb R\times (\mathbb R^3\setminus B)$, i.e. $\tilde g_0^{(4)}=-dt^2+\Sigma_i (dx^i)^2$. Since it is static, i.e. its twist tensor in the quotient formalism is zero, it is easy to verify that  Assumption 3.1 holds in this case (cf.\S7.5). So we can use the local chart $(\mathcal C,\tilde\Pi)$ in \S3 for the Bartnik boundary map at $[\tilde g^{(4)}_0]\in\mathbb E$. Obviously, the Bartnik data of this metic is
\begin{equation}
\tilde\Pi(\tilde g^{(4)}_0, 0)=(g_{S^2}, 2, 0, 0),
\end{equation}
where $g_{S^2}$ is the standard round metric on $S^2$. In this section we apply the ellipticity result proved in the previous sections to show that in a neighborhood of the standard flat boundary data $(g_{S^2}, 2, 0, 0)$, Bartnik boundary data admits unique stationary vacuum extensions up to diffeomorphisms.
\begin{theorem}
The kernel of $D\tilde\Pi_{(\tilde g_0^{(4)},0)}$ is trivial.
\end{theorem}
\begin{proof}
Assume that $(h^{(4)},G)\in\text{Ker}(D\tilde\Pi_{(\tilde g_0^{(4)},0)})$. Since $(h^{(4)},G)\in T_{(\tilde g_0^{(4)},0)}\mathcal C$, it must be a vacuum deformation, in the sense that the following equations hold on $M$:
\begin{equation}
\begin{cases}
(Ric)'_{h^{(4)}}=0\\
\Delta G=0.
\end{cases}
\end{equation}
In addition, since elements in $\mathcal C$ satisfy the gauge condition $\beta_{\tilde g_0^{(4)}}g^{(4)}=0$, the same equation holds for the deformation $h^{(4)}$:
\begin{equation}
\beta_{\tilde g_0^{(4)}}h^{(4)}=0\quad\text{on }M.
\end{equation}
Since $(h^{(4)}, G)$ preserves the Bartnik boundary data, linearization of the boundary equations are zero, i.e.
\begin{equation}
\begin{cases}
h_{\partial M}=0\\
(H_{\partial M})'_{h^{(4)}}=0\\
(tr_{\partial M}K)'_{h^{(4)}}+2G=0\\
(\omega_{\mathbf n})'_{h^{(4)}}+\nabla_{\partial M}G=0.
\end{cases}
\end{equation}
As we know, a stationary spacetime metric is uniquely determined by the data set $(g,X,N)$ on the hypersurface $M$, where $g$ is the induced metric on $M$, $X$ is the shift vector and $N$ is the lapse function. For the standard metric $\tilde g_0^{(4)}$, the corresponding data is $(g_0,0,1)$ with $g_0$ being the flat (Euclidean) metric on $\mathbb R^3\setminus B$. Thus the deformation $h^{(4)}$ can be decomposed as $h^{(4)}=(h,Y,v)$, where $h$ is the deformation of the Riemannian metric $g_0$, $Y$ is the deformation of the shift vector and $v$ is that of the lapse function.

The vacuum condition $Ric_{g^{(4)}}=0$ is equivalent to the following equations in terms of $(g,X,N)$ on $M$ (cf.[Mo]):
\begin{equation*}
\begin{cases}
K=-\frac{1}{2N}L_Xg\\
Ric_g+(trK)K-2K^2-\frac{1}{N}D^2N-\frac{1}{N}L_XK=0\\
\frac{1}{N}\Delta N+|K|^2-\frac{1}{N}tr(L_XK)=0\\
\delta K+d(trK)=0.	
\end{cases}
\end{equation*}

It is easy to linearize the equations above at $(g_0,0,1)$ and obtain a system in terms of $(h,Y,v)$, which is equivalent to equation (6.2), given by,
\begin{equation}
\begin{cases}
	Ric'_h-D^2v=0\\
	\Delta_{g_0} v=0\\
	\delta_{g_0}\delta_{g_0}^*Y-d\delta_{g_0} Y=0\\
	\Delta G=0.
\end{cases}
\quad\text{on}~M.
\end{equation}
The gauge equation (6.3) is equivalent to (cf.\S7.7 for the proof)
\begin{equation}
\begin{cases}
\delta_{g_0}Y=0\\
\delta_{g_0}h+\frac{1}{2}d(tr_{g_0}h+2v)=0,
\end{cases}\quad\text{on }M.
\end{equation}
The boundary conditions (6.4) are equivalent to:
\begin{equation}
\begin{cases}
	h_{\partial M}=0\\
	H'_h=0\\
	tr_{\partial M}\delta_{g_0}^*Y+2G=0\\
	[\delta_{g_0}^*Y(\mathbf n)]^T+\nabla_{g_0^T}G=0.
\end{cases}
\quad\text{on}~\partial M.
\end{equation}
In the last equation above, we use the superscript $'$$'$$^T$$'$$'$ to denote the restriction of tensors to the tangent bundle of $\partial M$. It is proved in [A2] that the first two equations in (6.5) combined with the first two boundary conditions in (6.7) imply that $v=0$ and $h=\delta_{g_0}^* Z$ for some vector field $Z\in C^{m+1,\alpha}_{\delta}(M)$ vanishing on $\partial M$. Additionally, $h$ must satisfy the gauge equation in (6.6). It follows that $\beta_{g_0}\delta_{g_0}^*Z=0$, which further implies that $Z=0$. So we obtain $h=0$ on $M$. 

It remains to prove $Y=0$ and $G=0$. The third equation in (6.5) and the first equation in (6.6) together imply:
\begin{equation*}
	\delta_{g_0}\delta_{g_0}^*Y=0\quad\text{on}~M.
\end{equation*}
Pair the equation above with $Y$. Then integration by parts gives,
\begin{equation}
	\begin{split}
		0&=\int_M \langle\delta_{g_0}\delta_{g_0}^*Y,Y\rangle_{g_0}dvol_{g_0}\\
		&=\int_M|\delta_{g_0}^*Y|^2-\int_{\partial M}\delta_{g_0}^*Y(\mathbf n,Y)-\int_{\infty}\delta^*_{g_0}Y(\mathbf n,Y)\\
		&=\int_M|\delta_{g_0}^*Y|^2-\int_{\partial M}\delta_{g_0}^*Y(\mathbf n,Y^T)-\int_{\partial M}\delta_{g_0}^*Y(\mathbf n,\mathbf n) Y^{\perp}.
			\end{split}
\end{equation}
In the boundary integral, we have decomposed $Y$ as $Y=Y^T+Y^{\perp}\mathbf n$, with $Y^{\perp}=\langle Y,\mathbf n\rangle_{g_0}$. In the second line, the boundary term at infinity $\int_{\infty}=\lim_{r\to\infty}\int_{S_r}$ is zero because the decay rate of $[\delta^*_{g_0}Y(\mathbf n,Y)]$ is $2\delta+1>2$.

For the second term in the last line of (6.8), one has,
\begin{equation*}
\begin{split}
\delta_{g_0}^*Y(\mathbf n,Y^T)&=\langle[\delta_{g_0}^*Y(\mathbf n)]^T,Y^T\rangle
=-\langle\nabla_{g_0^T}G,Y^T\rangle\\
&=-div_{g_0^T}(G\cdot Y^T)+G\cdot div_{g_0^T}Y^T\\
&=-div_{g_0^T}(G\cdot Y^T)-\frac{1}{2}(tr_{\partial M}\delta^*_{g_0}Y)\cdot div_{g_0^T}Y^T.
\end{split}
\end{equation*}
Here the second equality comes from the last boundary equation in (6.7) and the last equality is based on the third boundary equation in (6.7).
As for the last term in (6.8), notice that we have the following equality on the boundary:
\begin{equation*}
0=\delta_{g_0} Y=-\delta^*_{g_0}Y(\mathbf n,\mathbf n)-tr_{\partial M}\delta_{g_0}^*Y,
\end{equation*}
so that 
$\delta^*_{g_0}Y(\mathbf n,\mathbf n)=-tr_{\partial M}\delta_{g_0}^*Y$.
In addition,
$
tr_{\partial M}\delta^*_{g_0}Y
=tr_{\partial M}\delta_{g_0}^*Y^T+tr_{\partial M}\delta^*_{g_0}(Y^{\perp}\mathbf n)
=div_{g_0^T}Y^T+H_{g_0}Y^{\perp}
=div_{g_0^T}Y^T+2Y^{\perp}.
$
Substituting these computations into the integral equation (6.8) gives,
\begin{equation*}
	\begin{split}
		0&=\int_M|\delta_{g_0}^*Y|^2\\
		&\quad\quad+\int_{\partial M}\frac{1}{2}(div_{g_0^T}Y^T+2Y^{\perp})\cdot div_{g_0^T}Y^T+\int_{\partial M}(div_{g_0^T}Y^T+2Y^{\perp}) Y^{\perp}\\
		&=\int_M|\delta_{g_0}^*Y|^2+\frac{1}{2}\int_{\partial M}(div_{g_0^T}Y^T)^2+4Y^{\perp}\cdot div_{g_0^T}Y^T+4(Y^{\perp})^2\\
		&=\int_M|\delta_{g_0}^*Y|^2+\frac{1}{2}\int_{\partial M}(div_{g_0^T}Y^T+2Y^{\perp})^2\\
		&=\int_M|\delta_{g_0}^*Y|^2+\frac{1}{2}\int_{\partial M}(tr_{\partial M}\delta^*_{g_0}Y)^2.
	\end{split}
\end{equation*}
It immediately follows,
 \begin{equation}
 \begin{split}
 &\delta_{g_0}^*Y=0\quad\text{on }M,\\
 &tr_{\partial M}\delta_{g_0}^*Y=0\quad\text{on }\partial M.
 \end{split}
 \end{equation}
 The first equation above implies that $Y$ is a Killing vector field of the flat metric $g_0$ on $\mathbb R^3\setminus B$. In addition $Y$ must be asymptotically zero since it comes from a deformation of the asymptotically flat metrics in $\mathcal C$. Thus it follows $Y=0$ on $M$.
The boundary equation in (6.9) implies that $G=0$ on $\partial M$ according to (6.7). Furthermore, $G$ is asymptotically zero and harmonic according to (6.5). So $G=0$ on $M$. This completes the proof.
\end{proof}
Next, we prove that the Fredholm map $D\tilde\Pi_{(\tilde g_0^{(4)},0)}$ is of index 0 by showing the operator $D\mathcal F=(D\mathcal L,D\mathcal B)$ defined in \S4 has index 0 at $(\tilde g_0^{(4)},0)$. Here we use the idea in [A1] --- the boundary data in $D\mathcal B$ can be continuously deformed to a new collection of self-adjoint boundary data. 

Define the new boundary operator $D\mathcal{\tilde B}$ as follows:
\begin{equation}
\begin{split}
D\mathcal {\tilde B}: T_{(0,\tilde g_0^{(4)})}[\mathcal S\times C^{m,\alpha}_{\delta}(M)]\rightarrow \mathbb B\\
D\mathcal {\tilde B}(h^{(4)},G)= (\quad
&h_{\partial M},\\
&\nabla_{\mathbf n}(h^{(4)}(\mathbf n,\mathbf n)),\\
&\mathbf n(G),\\
&-\frac{1}{2}\nabla_{\mathbf n}[h^{(4)}(\partial_t)]^T,\\
&-\frac{1}{2}\nabla_{\mathbf n}h^{(4)}(\partial_t,\partial_t),\\
&-\nabla_{\mathbf n}h^{(4)}(\mathbf n)^T,\\
&-\nabla_{\mathbf n}h^{(4)}(\partial_t,\mathbf n)
\quad).
\end{split}
\end{equation}
Let $\mathcal N$ denote the space of deformations $(h^{(4)},G)$ of $(\tilde g_0^{(4)},0)$ in $\mathcal S\times C^{m,\alpha}_{\delta}(M)$ that are in the kernel of the boundary operator $D\mathcal{\tilde B}$, i.e.
\begin{equation*}
\begin{split}
\mathcal N=\{~(h^{(4)},G)\in T_{(0,\tilde g_0^{(4)})}[ \mathcal S\times C_{\delta}^{m,\alpha}(M)]:~~D\mathcal{\tilde B}(h^{(4)},G)=0\quad\}.
\end{split}
\end{equation*}
\begin{lemma}
The operator $D\mathcal L: \mathcal N\rightarrow [(S_2)_{\delta+2}^{m-2,\alpha}\times C_{\delta+2}^{m-2,\alpha}](M)$, given by
\begin{equation*}
\begin{split}
D\mathcal L(h^{(4)},G)=(D^*_{\tilde g_0^{(4)}}D_{\tilde g_0^{(4)}}h^{(4)},\Delta G),
\end{split}
\end{equation*}
is formally self-adjoint.
\end{lemma}
\begin{proof}
Let $(h^{(4)},G),(k^{(4)},J)$ denote two deformations in $\mathcal N$. Integration by parts yields:
\begin{equation*}
\begin{split}
&\int_{M}\langle D\mathcal L(h^{(4)},G),(k^{(4)},J)\rangle_{\tilde g^{(4)}_0}dvol_{g_0}
=\int_{M}\langle D\mathcal L(k^{(4)},J),(h^{(4)},G)\rangle_{\tilde g^{(4)}_0}\\
&\quad\quad\quad\quad\quad+\int_{\partial M}B[(k^{(4)},J),(h^{(4)},G)]-B[(h^{(4)},G),(k^{(4)},J)].
\end{split}
\end{equation*}
Here the boundary term at infinity is zero because of the decay behavior of the the deformations. The bilinear form $B$ is given by,
\begin{equation*}
B[(k^{(4)},J),(h^{(4)},G)]=\langle\nabla_{\mathbf n}k^{(4)},h^{(4)}\rangle_{\tilde g_0^{(4)}}+\mathbf n(J)G.
\end{equation*}
It is easy to verify that the terms above are zero because $(h^{(4)},G)$ and $(k^{(4)},J)$ make all the boundary terms listed in (6.10) vanish. Therefore $D\mathcal L$ is formally self-adjoint.
\end{proof}

Thus it follows that the new operator $(D\mathcal L,D\mathcal{\tilde B})$ is of index 0.
Next we show that the boundary data in $D\mathcal B$ can be deformed continuously through elliptic boundary data to $D\mathcal{\tilde B}$. Define a family of boundary operator $D\mathcal B_t,~t\in[0,1]$ as follows,
\begin{equation*}
\begin{split}
D\mathcal B_t: T_{(\tilde g_0^{(4)},0)}[&\mathcal S\times C^{m,\alpha}_{\delta}(M)]\rightarrow \mathbb B\\
D\mathcal B_t(h^{(4)},G)= (\quad
&h_{\partial M},\\
&(1-t)(H_{\partial M})'_{h^{(4)}}+t\nabla_{\mathbf n}(h^{(4)}(\mathbf n,\mathbf n)),\\
&(1-t)(tr_{\partial M}K)'_{h^{(4)}}+t\mathbf n(G),\\
&-\frac{1}{2}[\nabla_{\mathbf n}[h^{(4)}(\partial_t)](e_i)+(1-t)\nabla_{e_i}[h^{(4)}(\partial_t)](\mathbf n)]+(1-t)e_i(G),\\
&-\frac{1}{2}\nabla_{\mathbf n}h^{(4)}(\partial_t,\partial_t)+(1-t)[\frac{1}{2}\nabla_{\mathbf n}tr_Mh+\delta h(\mathbf n)],\\
&-\nabla_{\mathbf n}h^{(4)}(\mathbf n)^T+(1-t)[-\nabla_{e_i}h^{(4)}(e_i)^T+\frac{1}{2}\nabla_{g_0^T}(trh^{(4)})],\\
&-\nabla_{\mathbf n}h^{(4)}(\mathbf n,\partial_t)-(1-t)\nabla_{e_i}h^{(4)}(e_i,\partial_t)\quad).
\end{split}
\end{equation*}
Here $\{e_i\},~i=2,3$ denotes a local orthonormal basis of $T(\partial M)$. It is easy to check that $D\mathcal B_1=D\mathcal{\tilde B}$. When $t=0$, it is obvious that the first three lines of the boundary data above give the first three of the linearized Bartnik boundary data. It is also easy to check that the fourth line above (at $t=0$) has the same principal part as the linearized Bartnik boundary term $(\omega_{\mathbf n}+d_{\partial M}G)'$. Moreover, the last three lines above are respectively the $\mathbf n$, tangential ($\partial M$) and $\partial_t$ components of the gauge term $\beta_{\tilde g_0^{(4)}}h^{(4)}$ when $t=0$. Therefore,  $D\mathcal B_0=D\mathcal B$.
\begin{lemma}
The operator $(D\mathcal L, D\mathcal B_t)$ is elliptic for $t\in[0,1]$.
\end{lemma}
\begin{proof}
One can carry out the same proof as in \S4. Since the shift vector and lapse function of $\tilde g_0^{(4)}$ are simply $X=0$ and $N=1$, the principal symbol of the interior operator is (cf.equation (4.2)),
\begin{equation*}
L(\xi)=(\xi_1^2+\xi_2^2+\xi_3^2)I_{11\times 11}.
\end{equation*}
Set $(X,N)=(0,1)$ in equation (4.4) to compute the principal matrix of boundary operator $D\mathcal B_0$. Make a linear combination with the matrix of $D\mathcal B_1$. One can derive the principal matrix of the boundary operator $D\mathcal B_t~(t\in[0,1])$ as,
\begin{equation*}
B_t(\xi)=
\begin{bmatrix}
0_{3\times 8}&
\resizebox{.08\textwidth}{!}{$
\begin{matrix}
1&0&0\\
0&1&0\\
0&0&1
\end{matrix}
$}
\\
(\tilde B_t)_{8\times8}&*
\end{bmatrix},
\end{equation*}
where $\tilde B_t$ is given by (up to a factor of $-32^{-1}$),
\begin{equation*}
\resizebox{.99\textwidth}{!}{$
\begin{bmatrix}
0&0&0&0&0&t\xi_1&-(1-t)\xi_2&-(1-t)\xi_3\\
-t\xi_1&0&0&(1-t)\xi_2&(1-t)\xi_3&0&0&0\\
-2(1-t)\xi_2&0&(1-t)\xi_2&\xi_1&0&0&0&0\\
-2(1-t)\xi_3&0&(1-t)\xi_3&0&\xi_1&0&0&0\\
0&\xi_1&0&0&0&(1-t)\xi_1&2(1-t)\xi_2&2(1-t)\xi_3\\
0&(1-t)\xi_2&0&0&0&-(1-t)\xi_2&2\xi_1&0\\
0&(1-t)\xi_3&0&0&0&-(1-t)\xi_3&0&2\xi_1\\
0&0&\xi_1&(1-t)\xi_2&(1-t)\xi_3&0&0&0
\end{bmatrix}.$}
\end{equation*}
The determinant of $B_t(\xi)$ is
\begin{equation*}
\text{det}B_t(\xi)=\frac{1}{32}[t\xi_1^4-(2+t)(1-t)^2\xi_1^2(\xi_2^2+\xi_3^2)]\cdot[2(2+t)(1-t)^2(\xi_2^2+\xi_3^3)\xi_1^2-4t\xi_1^4]
.\end{equation*}
Let $\xi=z\mu+\eta$, where $z=i|\eta|^2$, the root of det$L(z\mu+\eta)=0$ with positive imaginary part. Then 
\begin{equation*}
\text{det}(B_t(z\mu+\eta))=-\frac{1}{32}[t+(2+t)(1-t)^2]\cdot[2(2+t)(1-t)^2+4t]|\eta|^8,
\end{equation*}
which obviously never vanishes for $t\in[0,1]$, $\eta\neq 0$. Thus the complementing boundary condition holds for all $t\in[0,1]$, which completes the proof.
\end{proof}

To conclude, we have the following theorem:
\begin{theorem}
The boundary map $\tilde\Pi$ is locally a diffeomorphism near $(\tilde g_0^{(4)},0)$.
\end{theorem}
\begin{proof}
From Lemma 6.2, 6.3 and the homotopy invariance of the index, it follows that the index of the boundary map $\tilde\Pi$ is 0 at $(\tilde g^{(4)}_0,0)$. In addition, it is proved in Theorem 6.1 that the kernel of $D\tilde\Pi_{(\tilde g_0^{(4)},0)}$ is trivial. Thus, the linearization $D\tilde\Pi_{(\tilde g_0^{(4)},0)}$ is an isomorphism. Then the inverse function theorem in Banach spaces gives the theorem.
\end{proof}

Now the equivalence relation between the maps $\tilde\Pi$ and $\Pi$ (cf.\S3) gives Theorem 1.2.

\section{Appendix}
\subsection{Notations}$~~$

On a Riemannian manifold $M\cong\mathbb R^3\setminus B^3$, we can pull back the standard coordinates $\{x^i\} (i=1,2,3)$ and the radius function $r$ from $\mathbb R^3\setminus B^3$ to $M$ under a chosen diffeomorphism. Then given $m\in\mathbb{N}$, and $\alpha,\delta\in\mathbb{R}$, the weighted H\"older spaces on $M$ are defined as,
\begin{equation*}
    \begin{split}
&C^m_{\delta}(M)=\{\text{functions}~v~\text{on}~M: ||v||_{C^m_{\delta}}=\Sigma_{k=0}^msup~r^{k+\delta}|\nabla^kv|<\infty\},\\
&C^{m,\alpha}_{\delta}(M)=\{\text{functions}~v~\text{on}~M:\\ &\quad\quad\quad\quad\quad||v||_{C^m_{\delta}}+sup_{x,y}[\text{min}(r(x),r(y))^{m+\alpha+\delta}
\frac{\nabla^mv(x)-\nabla^mv(y)}{|x-y|^{\alpha}}]
<\infty\},\\
&Met^{m,\alpha}_{\delta}(M)=\{\text{metrics}~g~\text{on}~M: (g_{ij}-\delta_{ij})\in C^{m,\alpha}_{\delta}(M)\},\\
&T^{m,\alpha}_{\delta}(M)=\{\text{vector fields}~X~\text{on}~M: X^i\in C^{m,\alpha}_{\delta}(M)\},\\
&(\wedge_1)^{m,\alpha}_{\delta}(M)=\{1-\text{forms}~\sigma~\text{on}~M: \sigma_{i}\in C^{m,\alpha}_{\delta}(M)\},\\
&D^{m,\alpha}_{\delta}(M)=\{\text{diffeomorphisms }\Psi:M\to M, \Psi(x^1,x^2,x^3)=(y^1,y^2,y^3): \\
&\quad\quad\quad\quad\quad\quad\quad\quad\quad\quad\quad\quad\quad\quad\quad\quad\quad\quad\quad\quad y^i-x^i\in C^{m,\alpha}_{\delta}(M)\}.
\end{split}
\end{equation*}
A tensor field is called asymptotically trivial (or zero) at the rate of $\delta$, if it belongs to one of the spaces above. It is well-known that Laplace-type operators with Dirichlet boundary condition are Fredholm when acting on these weighted H\"older spaces (cf.[LM],[Mc],[MV]).

On the compact manifold $\partial M$, we use the standard H\"older norm to define various Banach spaces of tensor fields as following,
\begin{equation*}
    \begin{split}
&C^m(\partial M)=\{\text{functions}~v~\text{on}~M: ||v||_{C^m}=\Sigma_{k=0}^msup~|\nabla^kv|<\infty\},\\
&C^{m,\alpha}(\partial M)=\{\text{functions}~v~\text{on}~M:~||v||_{C^m}+sup_{x,y}[
\frac{\nabla^mv(x)-\nabla^mv(y)}{|x-y|^{\alpha}}]
<\infty\},\\
&Met^{m,\alpha}(\partial M)=\{\text{metrics}~g~\text{on}~\partial M: g_{ab}\in C^{m,\alpha}(\partial M)\},\\
&(\wedge_1)^{m,\alpha}(\partial M)=\{1-\text{forms}~\sigma~\text{on}~\partial M: \sigma_{a}\in C^{m,\alpha}(\partial M)\}.
\end{split}
\end{equation*}
On the boundary $\partial M$ we use the index $1$ to denote the normal direction to $\partial M$ and $2,3$ the tangential direction. So in the definition above, the subscripts $a,b=2,3$. 

Let $(V^{(4)},g^{(4)})$ be a stationary spacetime such that $V^{(4)}\cong \mathbb R\times M$, with coordinates $\{t,x^i\}$ such that $\mathbb R$ is parametrized by $t$ and $M=\{t=0\}$. Assume $\partial_t$ is the time-like Killing vector field. In such a spacetime, a tensor field $\tau$ is called \textit{time-independent} if $L_{\partial_t}\tau=0$. In this case, we can think of $\tau$ as a tensor field on $M$. Define the following spaces of time-independent tensor fields in $(V^{(4)},g^{(4)})$:
\begin{equation*}
    \begin{split}
&C^{m,\alpha}_{\delta}(V^{(4)})=\{\text{functions }f\text{ in }V^{(4)}:~\partial_t f=0,~f\in C^{m,\alpha}_{\delta}(M)\},\\
&T^{m,\alpha}_{\delta}(V^{(4)})=\{\text{vector fields }Y\text{ in }V^{(4)}:~L_{\partial_t}Y=0,~Y^{\gamma}\in C^{m,\alpha}_{\delta}(M)\},\\
&S^{m,\alpha}_{\delta}(V^{(4)})=\{\text{symmetric 2-tensor fields }h^{(4)}\text{ in }V^{(4)}:~L_{\partial_t}h^{(4)}=0,~h^{(4)}_{\gamma\beta}\in C^{m,\alpha}_{\delta}(M)\},\\
&(\wedge_1)^{m,\alpha}_{\delta}(V^{(4)})=\{\text{1-forms }\omega\text{ in }V^{(4)}:~L_{\partial_t}\omega=0,~\omega_{\gamma}\in C^{m,\alpha}_{\delta}(M)\}.
\end{split}
\end{equation*}
In the definition above, the greek letters $\gamma,\beta=0,1,2,3$, where the index $0$ denotes the $\partial_t$ component of the tensor field and $1,2,3$ the tangential (to $M$) components. 
\subsection{Scalar fields $a,b$ in the time translation}$~~$

As in proposition 2.1, set up the diffeomorphism 
\begin{equation*}
\begin{split}
&\Phi_f: V^{(4)}\to V^{(4)}\\
&\Phi_f(t,x_1,x_2,x^3)=(t+f,x^1,x^2,x^3).
\end{split}
\end{equation*}
It maps the hypersurface $M=\{t=0\}$ to a new slice $\Phi_f(M)=\{t=f\}$ in $V^{(4)}$. Recall that $\mathbf N$ is the future-pointing time-like unit normal vector to the slice $M\subset (V^{(4)},g^{(4)})$ and $\mathbf n$ is the outward unit normal to $\partial M\subset (M,g)$. Their correspondence in the pull-back spacetime $(V^{(4)},\Phi_f^*(g^{(4)}))$ are $\hat{\mathbf N}$ and $\hat{\mathbf n}$. Thus $d\Phi_f(\hat{\mathbf N})$ and $d\Phi_f(\hat{\mathbf n})$ are the unit normal vector fields to the new slice $\Phi_f(M)$ and its boundary $\Phi_f(\partial M)$ in $(V^{(4)},g^{(4)})$ respectively.
Since $f|_{\partial M}=0$, $\partial M=\Phi_f(\partial M)$ and $d\Phi_f|_{\partial M}=Id_{\partial M}$. So along the boundary $\partial M$ the unit normal vectors $(\mathbf{\hat N},\mathbf{\hat n})$ must be mapped to the same subspace as $(\mathbf N,\mathbf n)$ in $TV^{(4)}$, i.e.
\begin{equation*}
\begin{split}
d\Phi_f(\mathbf{\hat N}),d\Phi_f(\mathbf{\hat n})\in \text{span}\{\mathbf N,\mathbf n\}.
\end{split}
\end{equation*}
Therefore there exist scalar fields $a,b,c,d$ belongs to $C^{m,\alpha}(\partial M)$ so that 
\begin{equation}
\begin{cases}
d\Phi_f(\mathbf{\hat N})=a\mathbf N+b\mathbf n,\\
d\Phi_f(\mathbf{\hat n})=c\mathbf N+d\mathbf n.
\end{cases}
\end{equation}
In addition, by the definition of pull-back metric we have 
\begin{equation*}
\begin{split}
&\langle d\Phi_f(\mathbf{\hat N}),d\Phi_f(\mathbf{\hat N})\rangle_{g{(4)}}=\langle\hat{\mathbf N},\hat{\mathbf N}\rangle_{\Phi_f^*g^{(4)}}=-1;\\
&\langle d\Phi_f(\mathbf{\hat n}),d\Phi_f(\mathbf{\hat n})\rangle_{g{(4)}}=\langle\hat{\mathbf n},\hat{\mathbf n}\rangle_{\Phi_f^*g^{(4)}}=1;\\
&\langle d\Phi_f(\mathbf{\hat N}),d\Phi_f(\mathbf{\hat n})\rangle_{g{(4)}}=\langle\hat{\mathbf N},\hat{\mathbf n}\rangle_{\Phi_f^*g^{(4)}}=0.
\end{split}
\end{equation*}
So we obtain the following equations for $(a,b,c,d)$, 
\begin{equation}
\begin{cases}
-a^2+b^2=-1,\\
-c^2+d^2=1,\\
-ac+bd=0.
\end{cases}
\end{equation}
It further implies that $a^2=d^2$ and $b^2=c^2$. Without loss of generality (up to the choice of directions), we can assume,
$$a=d>0,~b=c>0.$$

Moreover, the vector field $d\Phi_f(\mathbf{\hat N})$ must be orthogonal to $d\Phi_f(\partial_{x^i})$ with respect to $g^{(4)}$ because
\begin{equation}
\begin{split}
\langle d\Phi_f(\mathbf{\hat N}),d\Phi_f(\partial_{x^i})\rangle_{g^{(4)}}=\langle \mathbf{\hat N},\partial_{x^i}\rangle_{\Phi^*g^{(4)}}=0,~\forall i=1,2,3.
\end{split}
\end{equation}
By the definition of $\Phi_f$, it follows that $d\Phi_f(\partial_{x^i})=(\partial_if)\partial_t+\partial_{x^i}$. On the other hand, it is easy to verify that,
\begin{equation}\langle \partial_t-X+N^2\nabla f,(\partial_if)\partial_t+\partial_{x^i}\rangle_{g^{(4)}}=0,~\forall i=1,2,3.
\end{equation}
where $\nabla f$ denotes the gradient of $f$ with respect to the metric $g^{(4)}$. Thus, equations (7.3) and (7.4) imply that $d\Phi_f(\mathbf{\hat N})$ must be proportional to $\partial_t-X+N^2\nabla f$. We make the direction choice so that $d\Phi_f(\mathbf{\hat N})$ is future pointing, then
\begin{equation*}
\begin{split}
d\Phi_f(\mathbf{\hat N})&=\frac{\partial_t-X+N^2\nabla f}{\sqrt{-||\partial_t-X+N^2\nabla f||^2_{g^{(4)}}}}\\
&=\frac{\partial_t-X+N^2\nabla f}{N\sqrt{1+2X(f)-||N^2\nabla f||^2_{g^{(4)}}}}\\
&=\frac{\partial_t-X+N^2\nabla_{g} f}{N\sqrt{1+2X(f)+X(f)^2-N^2||\nabla_{g} f||^2_g}}\\
&=\frac{\partial_t-X+N^2\nabla_{g} f}{N\sqrt{(1+X(f))^2-N^2||\nabla_{g} f||^2}_g},
\end{split}
\end{equation*}
where $\nabla_{g}f$ denotes the gradient of $f$ with respect to the induced metric $g$ on $M$.
Therefore, according to the first equation in (7.1), we obtain
\begin{equation*}
\begin{split}
a&=-\langle\mathbf{N},d\Phi_f(\mathbf{\hat N})\rangle_{g^{(4)}}\\
&=-\frac{g^{(4)}(\partial_t-X,\partial_t-X+N^2\nabla_g f)}{N^2\sqrt{(1+X(f))^2-N^2||\nabla_{g} f||^2}}\\
&=\frac{1+X(f)}{\sqrt{(1+X(f))^2-N^2||\nabla_g f||^2}}.
\end{split}
\end{equation*}
In the above, the second equality is because $\mathbf N=(\partial_t-X)/N$ according to the 3+1 formalism (2.1). Moreover, since $f$ is chosen to be vanishing on $\partial M$, so $\nabla_gf=\mathbf n(f)\cdot\mathbf n$ on the boundary. Thus $||\nabla_g f||_g=\mathbf n(f)$ on the boundary $\partial M$, and $X(f)|_{\partial M}=\langle X,\mathbf n\rangle \mathbf n(f)$ and consequently,
\begin{equation*}
\begin{split}
a=\frac{1+\langle X,\mathbf n\rangle \mathbf n(f)}{\sqrt{[1+\langle X,\mathbf n\rangle \mathbf n(f)]^2-N^2|\mathbf n(f)|^2}}\quad\text{on }\partial M,
\end{split}
\end{equation*}
which is the formula (2.11). Based on (7.2), we easily derive the formula for $b$ as follows,
\begin{equation*}
\begin{split}
b=\frac{N\mathbf n(f)}{\sqrt{[1+\langle X,\mathbf n\rangle \mathbf n(f)]^2-N^2|\mathbf n(f)|^2}}.
\end{split}
\end{equation*}
\subsection{Linearization of boundary operator $\mathcal B$}$~~$

For simplicity of notation, we will write $h$ instead of $h^{(4)}$ in this section. Subindex 0 denotes the $\partial_t$ direction in $V^{(4)}$, index 1 denotes the outward normal direction to $\partial M$ and $2,3$ denote the tangential directions on $\partial M$.
\\
\\
1.With respect to the deformation $h$, linearization of $g_{\partial M}$ is easily seen to be:
\begin{equation*}
(g_{\partial M})'_h=(h_{22},h_{23},h_{33}).
\end{equation*}
\\
2.Linearization of $H_{\partial M}$:

By the formula of the linearization of mean curvature (cf.\S 5.2 in [Az] for example), one has
\begin{equation*}
\begin{split}
2(H_{\partial M})'_h=-2\partial_2h_{12}-2\partial_3h_{13}+\partial_1(h_{22}+h_{33})+O_0(h).
\end{split}
\end{equation*}
\\
3.Linearization of the second fundamental form $K$:

The defining equation for $K$ is
$$K_{ij}=-\frac{1}{2N}L_{X}g_{ij},$$
where $g_{ij}$ denotes the Riemannian metric induced from $g^{(4)}$ on $M$ and $X$ is the shift vector on $M$. Notice that $X$ is the dual of the $1-$form $g^{(4)}(\partial_t)|_{M}$, i.e. $X^{\flat}_i=g^{(4)}_{0i}$. 
Variation of $K$ with respect to $h$ is given by,
$$(K'_h)_{ij}=-\frac{1}{2N}(L_{(X)'}g_{ij}+L_{X}h_{ij})+O_0(h).$$
As for the variation $X'$, it is given by,
\begin{equation*}
\begin{split}
X^i&=g^{ik}g^{(4)}_{0k},\\
(X')^i&=\tilde h^{ik}g^{(4)}_{0k}+g^{ik}h_{0k},
\end{split}
\end{equation*}
where $\tilde h$ is the variation of the inverse $g^{ij}$. It is easy to see that
$$\tilde h^{ij}g_{jk}=-g^{ij}h_{jk}.$$
Therefore,
\begin{equation*}
\begin{split}
L_{X'}g_{ij}&=g_{ik}\nabla^M_j (X')^{k}+g_{jk}\nabla^M_i (X')^{k}\\
&=\nabla^M_j\{g_{ik} (X')^{k}\}+\nabla^M_i\{g_{jk} (X')^{k}\}\\
&=\nabla^M_j\{g_{il} \tilde h^{lk}g^{(4)}_{0k}+g_{il}g^{lk}h_{0k}\}+\nabla^M_i\{g_{jl} \tilde h^{lk}g^{(4)}_{0k}+g_{jl}g^{lk}h_{0k}\}\\
&=\nabla^M_j\{-h_{il} g^{lk}g^{(4)}_{0k}+h_{0i}\}+\nabla^M_i\{-h_{jl} g^{lk}g^{(4)}_{0k}+h_{0j}\}\\
&=\nabla^M_j\{-h_{il} X^l+h_{0i}\}+\nabla^M_i\{-h_{jl}X^l+h_{0j}\},
\end{split}
\end{equation*}
and
\begin{equation*}
\begin{split}
L_Xh_{ij}&=\nabla^M_{X}h_{ij}+h_{ik}\nabla^M_{j} X^{k}+h_{jk}\nabla^M_{i} X^{k}=\nabla^M_{X}h_{ij}+O_0(h).
\end{split}
\end{equation*}
In the above the connection $\nabla^M$ denotes the covariant derivative with respect to the induced metric $g$ on $M$.
Thus,
\begin{equation*}
\begin{split}
(K'_h)_{ij}=-\frac{1}{2N}[\partial_i h_{0j}+\partial_j h_{0i}+\partial_{X}h_{ij}-X^l(\partial_ih_{jl}+\partial_jh_{il})]+O_0(h),
\end{split}
\end{equation*}
and consequently,
\begin{equation*}
\begin{split}
[tr_{\partial M}K]'_h&=tr_{\partial M}(K'_h)+O_0(h)\\
&=-\frac{1}{2N}[2\partial_2 h_{02}+2\partial_3 h_{03}+\partial_{X}(h_{22}+h_{33})-X^l(2\partial_2h_{2l}+2\partial_3h_{3l})]+O_0(h),\\
\big((\omega_{\mathbf n})'_h\big)_i&=[(K(\mathbf n)|_{\partial M})'_h]_i\\
&=[K'_h(\mathbf n)|_{\partial M}]_i+O_0(h)\\
&=-\frac{1}{2N}[\partial_1 h_{0i}+\partial_i h_{01}+\partial_{X}h_{1i}-X^l(\partial_1h_{il}+\partial_ih_{1l})]+O_0(h),\\
\text{with}~i=2,3.
\end{split}
\end{equation*}
\\
4. Linearization of the gauge term $\beta_{\tilde g^{(4)}}g^{(4)}$.\\

Obviously $[\beta_{\tilde g^{(4)}}g^{(4)}]'_h=\beta_{\tilde g^{(4)}}h$. Take an arbitrary vector field $Y\in T^{m,\alpha}_{\delta}(V^{(4)})$. Then $\beta_{g^{(4)}}h(Y)=\delta_{g^{(4)}}h(Y)+\frac{1}{2}Y(trh),$ where the two terms on the right side are computed as,
\begin{equation*}
\begin{split}
\delta_{g^{(4)}} h(Y)&=\nabla_{\mathbf N}h(\mathbf N, Y)-\Sigma_{k=1,2,3}\nabla_k h(Y)_k+O_0(h)\\
&=\frac{1}{N^2}\nabla_{\partial_t-X} h(\partial_t-X,Y)-\Sigma_{k=1,2,3}\nabla_k h(Y)_k+O_0(h)\\
&=-\frac{1}{N^2}\partial_{X} h(\partial_t-X,Y)-\Sigma_{k=1,2,3}\partial_k h(Y)_k+O_0(h),\\
trh&=-h(\mathbf N,\mathbf N)+h_{11}+h_{22}+h_{33}\\
&=-\frac{1}{N^2}h(\partial_t-X,\partial_t-X)+h_{11}+h_{22}+h_{33}\\
&=-\frac{1}{N^2}(h_{00}+X^iX^jh_{ij}-2X^lh_{0l})+h_{11}+h_{22}+h_{33}.
\end{split}
\end{equation*}
Therefore, for $i=1,2,3$, the $i$th component of the linearized gauge term is given by,
\begin{equation*}
\begin{split}
[\beta_{\tilde g^{(4)}}h^{(4)}]_i=&-\frac{1}{2N^2}[\partial_ih_{00}+X^kX^j\partial_ih_{kj}-2X^l\partial_ih_{0l}]+\frac{1}{2}\partial_i(h_{11}+h_{22}+h_{33})\\
&-\frac{1}{N^2}[\partial_{X} h_{0i}-X^k\partial_Xh_{ki}]-\Sigma_{k=1,2,3}\partial_k h_{ki}+O_0(h);
\end{split}
\end{equation*}
and the $\partial_t$ component of the gauge term is given by,
\begin{equation*}
\begin{split}
[\beta_{\tilde g^{(4)}}h^{(4)}]_0=-\frac{1}{N^2}[\partial_{X} h_{00}-X^k\partial_Xh_{k0}]-\Sigma_{k=1,2,3}\partial_k h_{k0}+O_0(h).
\end{split}
\end{equation*}

Summing up all the computations above, we obtain the boundary symbol matrix $\tilde B$ in equation (4.4), given by,
\begin{equation*}
\begin{split}
&\tilde B=\\
&\resizebox{.99\textwidth}{!}{$-\frac{1}{2N^2}
\begin{bmatrix}
0&0&0&0&0~&0&2N^2\xi_2&2N^2\xi_3\\
0&0&0&2N\xi_2&2N\xi_3~&0&-2N\xi_2X^1&-2N\xi_3X^1\\
-2N^3\xi_2&0&N\xi_2&N\xi_1&0~&-N\xi_2X^1&N\xi_3X^3&-N\xi_2X^3\\
-2N^3\xi_3&0&N\xi_3&0&N\xi_1~&-N\xi_3X^1&-N\xi_3X^2&N\xi_2X^2\\
0&\xi_1&2S-2\xi_1X^1&-2\xi_1X^2&-2\xi_1X^3~&\xi_1X^1X^1+N^2\xi_1-2S X^1&\xi_1X^1X^2-2S X^2+2N^2\xi_2&\xi_1X^1X^3-2S X^3+2N^2\xi_3\\
0&\xi_2&-2\xi_2X^1&2S-2\xi_2X^2&-2\xi_2X^3~&\xi_2X^1X^1-N^2\xi_2&\xi_2X^1X^2-2S X^1+2N^2\xi_1&\xi_2X^1X^3\\
0&\xi_3&-2\xi_3X^1&-2\xi_3X^2&2S-2\xi_3X^3~&\xi_3X^1X^1-N^2\xi_3&\xi_3X^1X^2&\xi_3X^1X^3-2S X^1+2N^2\xi_1\\
0&2S&2(N^2\xi_1-SX^1)&2(N^2\xi_2-SX^2)&2(N^2\xi_3-SX^3)&0&0&0
\end{bmatrix}$},
\end{split}
\end{equation*}
inside which $S=\xi_1X^1+\xi_2X^2+\xi_3X^3$. 
\subsection{Calculation of the boundary symbol matrix }$~~~$

To compute the determinant of the matrix $\tilde B$, we can first simplify it by factoring out the common factors in every row: factor out $(-2N^2)$ from the first row, $2N$ from the second row, $N$ from the third and fourth row, and $2$ from the last row. Then we obtain that $\text{det}\tilde B=-\frac{1}{32N^{11}}\text{det}\hat B$, with $\hat B$ given by,
$$
\hat B=\resizebox{.97\textwidth}{!}{$
\begin{bmatrix}
0&0&0&0&0~&0&-\xi_2&-\xi_3\\
0&0&0&\xi_2&\xi_3~&0&-\xi_2X^1&-\xi_3X^1\\
-2N^2\xi_2&0&\xi_2&\xi_1&0~&-\xi_2X^1&\xi_3X^3&-\xi_2X^3\\
-2N^2\xi_3&0&\xi_3&0&\xi_1~&-\xi_3X^1&-\xi_3X^2&\xi_2X^2\\
0&\xi_1&2S-2\xi_1X^1&-2\xi_1X^2&-2\xi_1X^3~&\xi_1X^1X^1+N^2\xi_1-2S X^1&\xi_1X^1X^2-2S X^2+2N^2\xi_2&\xi_1X^1X^3-2S X^3+2N^2\xi_3\\
0&\xi_2&-2\xi_2X^1&2S-2\xi_2X^2&-2\xi_2X^3~&\xi_2X^1X^1-N^2\xi_2&\xi_2X^1X^2-2S X^1+2N^2\xi_1&\xi_2X^1X^3\\
0&\xi_3&-2\xi_3X^1&-2\xi_3X^2&2S-2\xi_3X^3~&\xi_3X^1X^1-N^2\xi_3&\xi_3X^1X^2&\xi_3X^1X^3-2S X^1+2N^2\xi_1\\
0&S&N^2\xi_1-SX^1&N^2\xi_2-SX^2&N^2\xi_3-SX^3~&0&0&0
\end{bmatrix}$},
$$
We now carry out the following row and column operation to simplify $\hat B$. First, multiply the first row of $\hat B$ by $-X^1$ and then add it to the second row. Multiply the first row by $2N^2$ and then add it to the fifth row. The matrix becomes:
$$
\hat B_1=\resizebox{.97\textwidth}{!}{$
\begin{bmatrix}
0&0&0&0&0~&0&-\xi_2&-\xi_3\\
0&0&0&\xi_2&\xi_3~&0&0&0\\
-2N^2\xi_2&0&\xi_2&\xi_1&0~&-\xi_2X^1&\xi_3X^3&-\xi_2X^3\\
-2N^2\xi_3&0&\xi_3&0&\xi_1~&-\xi_3X^1&-\xi_3X^2&\xi_2X^2\\
0&\xi_1&2S-2\xi_1X^1&-2\xi_1X^2&-2\xi_1X^3~&\xi_1X^1X^1+N^2\xi_1-2S X^1&\xi_1X^1X^2-2S X^2&\xi_1X^1X^3-2S X^3\\
0&\xi_2&-2\xi_2X^1&2S-2\xi_2X^2&-2\xi_2X^3~&\xi_2X^1X^1-N^2\xi_2&\xi_2X^1X^2-2S X^1+2N^2\xi_1&\xi_2X^1X^3\\
0&\xi_3&-2\xi_3X^1&-2\xi_3X^2&2S-2\xi_3X^3~&\xi_3X^1X^1-N^2\xi_3&\xi_3X^1X^2&\xi_3X^1X^3-2S X^1+2N^2\xi_1\\
0&S&N^2\xi_1-SX^1&N^2\xi_2-SX^2&N^2\xi_3-SX^3~&0&0&0
\end{bmatrix}$}.
$$
In $\hat B_1$, multiply the second row by $(-N^2)$ and add it to the last row:
$$
\hat B_2=\resizebox{.97\textwidth}{!}{$
\begin{bmatrix}
0&0&0&0&0~&0&-\xi_2&-\xi_3\\
0&0&0&\xi_2&\xi_3~&0&0&0\\
-2N^2\xi_2&0&\xi_2&\xi_1&0~&-\xi_2X^1&\xi_3X^3&-\xi_2X^3\\
-2N^2\xi_3&0&\xi_3&0&\xi_1~&-\xi_3X^1&-\xi_3X^2&\xi_2X^2\\
0&\xi_1&2S-2\xi_1X^1&-2\xi_1X^2&-2\xi_1X^3~&\xi_1X^1X^1+N^2\xi_1-2S X^1&\xi_1X^1X^2-2S X^2&\xi_1X^1X^3-2S X^3\\
0&\xi_2&-2\xi_2X^1&2S-2\xi_2X^2&-2\xi_2X^3~&\xi_2X^1X^1-N^2\xi_2&\xi_2X^1X^2-2S X^1+2N^2\xi_1&\xi_2X^1X^3\\
0&\xi_3&-2\xi_3X^1&-2\xi_3X^2&2S-2\xi_3X^3~&\xi_3X^1X^1-N^2\xi_3&\xi_3X^1X^2&\xi_3X^1X^3-2S X^1+2N^2\xi_1\\
0&S&N^2\xi_1-SX^1&-SX^2&-SX^3~&0&0&0
\end{bmatrix}$}.
$$
In $\hat B_2$, multiply the second column by $N^2$ and add it to the sixth column. Then multiply the second column by $X^i$ and add it to the $(2+i)$th column ($i=1,2,3$):
$$
\hat B_3=\resizebox{.97\textwidth}{!}{$
\begin{bmatrix}
0&0&0&0&0~&0&-\xi_2&-\xi_3\\
0&0&0&\xi_2&\xi_3~&0&0&0\\
-2N^2\xi_2&0&\xi_2&\xi_1&0~&-\xi_2X^1&\xi_3X^3&-\xi_2X^3\\
-2N^2\xi_3&0&\xi_3&0&\xi_1~&-\xi_3X^1&-\xi_3X^2&\xi_2X^2\\
0&\xi_1&2S-\xi_1X^1&-\xi_1X^2&-\xi_1X^3~&\xi_1X^1X^1+2N^2\xi_1-2S X^1&\xi_1X^1X^2-2S X^2&\xi_1X^1X^3-2S X^3\\
0&\xi_2&-\xi_2X^1&2S-\xi_2X^2&-\xi_2X^3~&\xi_2X^1X^1&\xi_2X^1X^2-2S X^1+2N^2\xi_1&\xi_2X^1X^3\\
0&\xi_3&-\xi_3X^1&-\xi_3X^2&2S-\xi_3X^3~&\xi_3X^1X^1&\xi_3X^1X^2&\xi_3X^1X^3-2S X^1+2N^2\xi_1\\
0&S&N^2\xi_1&0&0~&N^2S&0&0
\end{bmatrix}$}.
$$
In $\hat B_3$, multiply the $i$th column by $X^1$ and add it to the $(i+3)$th column ($i=3,4,5$). Then multiply the second column by $X^i$ and add it to column $(i+1),~(i=1,2,3)$.
$$
\hat B_4=\resizebox{.97\textwidth}{!}{$
\begin{bmatrix}
0&0&0&0&0~&0&-\xi_2&-\xi_3\\
0&0&0&\xi_2&\xi_3~&0&\xi_2X^1&\xi_3X^1\\
-2N^2\xi_2&0&\xi_2&\xi_1&0~&0&\xi_3X^3+\xi_1X^1&-\xi_2X^3\\
-2N^2\xi_3&0&\xi_3&0&\xi_1~&0&-\xi_3X^2&\xi_2X^2+\xi_1X^1\\
0&\xi_1&2S-\xi_1X^1&-\xi_1X^2&-\xi_1X^3~&2N^2\xi_1&-2S X^2&-2S X^3\\
0&\xi_2&-\xi_2X^1&2S-\xi_2X^2&-\xi_2X^3~&0&2N^2\xi_1&0\\
0&\xi_3&-\xi_3X^1&-\xi_3X^2&2S-\xi_3X^3~&0&0&2N^2\xi_1\\
0&S&N^2\xi_1&0&0~&N^2S+N^2\xi_1X^1&0&0
\end{bmatrix}$}.
$$
In $\hat B_4$, multiply column $2$ by $X^i$ and add it to column $(i+2)$, $(i=1,2,3)$.  Multiply column 1 by $(2N^2)^{-1}$ and add it to column 3. Then multiply the first row by $X^1$ and add it to row 2:
$$
\hat B_5=\resizebox{.97\textwidth}{!}{$
\begin{bmatrix}
0&0&0&0&0~&0&-\xi_2&-\xi_3\\
0&0&0&\xi_2&\xi_3~&0&0&0\\
-2N^2\xi_2&0&0&\xi_1&0~&0&\xi_3X^3+\xi_1X^1&-\xi_2X^3\\
-2N^2\xi_3&0&0&0&\xi_1~&0&-\xi_3X^2&\xi_2X^2+\xi_1X^1\\
0&\xi_1&2S&0&0~&2N^2\xi_1&-2S X^2&-2S X^3\\
0&\xi_2&0&2S&0~&0&2N^2\xi_1&0\\
0&\xi_3&0&0&2S~&0&0&2N^2\xi_1\\
0&S&N^2\xi_1+SX^1&SX^2&SX^3~&N^2S+N^2\xi_1X^1&0&0
\end{bmatrix}$}.
$$
This is the matrix given in (4.5).
\subsection{Calculation in the projection formalism}$~~$

Take a general stationary metric in $V^{(4)}$ expressed in the projection formalism as,
$$g^{(4)}=-u^2(dt+\theta)^2+g_S.$$
In this section, we use $\nabla$ to denote the Levi-Civita connection of $g^{(4)}$, and $\nabla_{g_S}$ to denote that of $g_S$. We first state two simple facts.

1. Since $\partial_t$ is a Killing vector field, it follows that for any vector field $Y\in TV^{(4)}$, we have
$\langle\nabla_{\partial t}\partial_t, Y\rangle=-\langle\nabla_Y\partial_t,\partial_t\rangle=uY(u)$.
Thus,
\begin{equation}
\nabla_{\partial_t}\partial_t=u\nabla u.
\end{equation}
Notice that $\nabla u$ is a vector field on $S$ because $u$ is independent of $t$.

2. For any horizontal vector fields $v,w\in TS$, one has $\langle v,\partial_t\rangle=0$, $L_{\partial t}v=0$, and hence
$$\langle\nabla_vw,\partial_t\rangle=-\langle w,\nabla_v\partial_t\rangle=\langle v,\nabla_w\partial_t\rangle=-\langle \nabla_wv,\partial_t\rangle.$$
Let $\xi=-u^2(dt+\theta)$ be the dual of $\partial_t$. Then by the definition of exterior derivative, we have 
$d\xi(v,w)=\langle\nabla_v\partial_t,w\rangle-\langle\nabla_w\partial_t,v\rangle$. Combining the quality above, we obtain
$d\xi(v,w)=2\langle\nabla_wv,\partial_t\rangle$. 
On the other hand, 
$d\xi=d[-u^2(dt+\theta)]=-u^2d\theta-2udu\wedge(dt+d\theta)=-u^2d\theta+2u^{-1}du\wedge\xi$.
Thus we can derive that,
\begin{equation}
\begin{split}
&2\langle\nabla_wv,\partial_t\rangle=d\xi(v,w)=-u^2d\theta(v,w)\\
&\xi([v,w])=\langle \nabla_vw-\nabla_wv,\partial_t\rangle=2\langle\nabla_vw,\partial_t\rangle=u^2d\theta(v,w).
\end{split}
\end{equation}
Next we give a proof for the formula (5.12):

Let $\alpha=dt+\theta=-u^{-2}\xi$, so $\alpha(\partial_t)=1,~\alpha(v)=0~\forall v\in TS$. Then according to the the following Lie-derivative formula for time-independent vector feilds $A,B,Y$ in the spacetime: 
\begin{equation*}
	\begin{split}
		L_{Y}\alpha^2(A,B)=Y[\alpha^2(A,B)]-\alpha^2([Y,A],B)-\alpha^2(A,[Y,B]),
	\end{split}
\end{equation*}
it is easy to see that 
\begin{equation*}
	\begin{cases}
		L_{ Y}\alpha^2(\partial t,\partial t)=0\\
		[L_{Y}\alpha^2]^T=0.
	\end{cases}
\end{equation*}
As for the mixed component of $L_Y\alpha^2$, we can carry out the following computation for $v\in TS$,
\begin{equation}
\begin{split}
L_Y\alpha^2(\partial_t,v)&=-\alpha^2([Y,v],\partial_t)\\
&=-\alpha([Y,v])\\
&=u^{-2}\xi([Y,v]).
\end{split}
\end{equation}
As discussed in \S5, any vector field $Y\in T^{m,\alpha}_{\delta}(V^{(4)})$ can be decomposed as,
\begin{equation}
Y=Y^T-\frac{Y^{\perp}}{u}\partial t,~\text{with}~Y^T\in TS \text{ and }Y^{\perp}=\frac{1}{u}\langle Y,\partial_t\rangle.
\end{equation}
Thus, for $v\in TS$, one has,
\begin{equation*}
\begin{split}
\xi[Y,v]&=\xi([Y^T,v])-\xi([\frac{Y^{\perp}}{u}\partial_t,v])=\xi([Y^T,v])+\xi[v(\frac{Y^{\perp}}{u})\partial_t]\\
&=u^2d\theta(Y^T,v)-u^2v(\frac{Y^{\perp}}{u}).
\end{split}
\end{equation*}
In the last equality above, we use the formula in (7.6) to compute $\xi([Y^T,v])$.
Plugging this to equation (7.7) we obtain
\begin{equation*}
[L_Y\alpha^2(\partial_t)]^T=d\theta(Y^T)-d(\frac{Y^{\perp}}{u})
\end{equation*}
This completes the proof of (5.12).

Using the same notation as above, we give a proof of the formula (5.13) as follows.

Based on the decomposition (7.8), we have
\begin{equation}
\begin{split}
2\delta^*_{g^{(4)}}Y=L_{Y^T}g^{(4)}-L_{\frac{Y^{\perp}}{u}\partial_t}g^{(4)}.
\end{split}
\end{equation}
In the following, we assume $v,w\in TS$. For the first term in (7.9), we have
\begin{equation*}
\begin{split}
&L_{Y^T}g^{(4)}(\partial_t,\partial_t)=2\langle\nabla_{\partial_t}Y^T,\partial_t\rangle=2\langle\nabla_{Y^T}\partial_t,\partial_t\rangle=-2uY^T(u),\\
&L_{Y^T}g^{(4)}(v,w)=\langle\nabla_vY^T,w\rangle+\langle\nabla_wY^T,v\rangle=L_{Y^T}g_S(v,w),\\
&L_{Y^T}g^{(4)}(\partial_t,v)=\langle\nabla_{\partial_t}Y^T,v\rangle+\langle\nabla_vY^T,\partial_t\rangle\\
&\quad\quad=\langle\nabla_{Y^T}\partial_t,v\rangle+\langle\nabla_vY^T,\partial_t\rangle=-\langle\nabla_{Y^T}v,\partial_t\rangle+\langle\nabla_vY^T,\partial_t\rangle\\
&\quad\quad=2\langle\nabla_vY^T,\partial_t\rangle=-u^2d\theta(Y^T,v).
\end{split}
\end{equation*}
In the last equality, we apply (7.6) to the term $2\langle\nabla_vY^T,\partial_t\rangle$.
Summing up the equations,
\begin{equation}
\begin{cases}
L_{Y^T}g^{(4)}(\partial_t,\partial_t)=-2uY^T(u)\\
[L_{Y^T}g^{(4)}(\partial_t)]^T=-u^2d\theta(Y^T)\\
[L_{Y^T}g^{(4)}]^T=L_{Y^T}g_S.
\end{cases}
\end{equation}
As for the second term on the right side of (7.9), basic calculation yields,
\begin{equation*}
\begin{split}
L_{\frac{Y^{\perp}}{u}\partial_t}g^{(4)}=\frac{Y^{\perp}}{u}L_{\partial_t}g^{(4)}+d(\frac{Y^{\perp}}{u})\odot\xi=d(\frac{Y^{\perp}}{u})\odot\xi.
\end{split}
\end{equation*}
Thus,
\begin{equation}
\begin{cases}
L_{\frac{Y^{\perp}}{u}\partial_t}g^{(4)}(\partial_t,\partial_t)=0\\
[L_{\frac{Y^{\perp}}{u}\partial_t}g^{(4)}(\partial_t)]^T=-u^2d(\frac{Y^{\perp}}{u})\\
[L_{\frac{Y^{\perp}}{u}\partial_t}g^{(4)}]^T=0.
\end{cases}
\end{equation}
Equations (7.10) and (7.11) together give (5.13).

At last we derive the decomposition (5.9) of the Bianchi gauge operator.

We assume $g^{(4)}$ is in addition vacuum, which is equivalent to the following system in the projection formalism, (cf.[G],[H1],[H2]),

\begin{equation}
\begin{cases}
Ric_{g_S}=\frac{1}{u}D^2_{g_S}u+2u^{-4}(\omega^2-|\omega|^2_{g_S}\cdot g_S)\\
\Delta_{g_S} u=2u^{-3}|\omega|^2_{g_S}\\
\delta_{g_S}\omega+3u^{-1}\langle du,\omega\rangle_{g_S}=0\\
d\omega=0
\end{cases},
\end{equation}
where $\omega$ is the twist tensor defined as,
$$\omega=-\frac{1}{2}u^{3}\star_{g_S}d\theta.$$
Here we use subscript $'$$'$$_{g_S}$$'$$'$ to denote geometric operators (connection and Laplacian) of the Riemannian metric $g_S$ on the quotient manifold $S$.
First observe that, from the last equation in (7.12), it follows that
\begin{equation*}
\begin{split}
0=d\omega=d(u^3\star_{g_S}d\theta)=d\star_{g_S} (u^3d\theta)=\delta_{g_S}(u^3d\theta)=u^3\delta_{g_S} d\theta-3u^2d\theta(\nabla u).
\end{split}
\end{equation*}
Thus, we obtain
\begin{equation}
\begin{split}
u\delta_{g_S} d\theta=3d\theta(\nabla u).
\end{split}
\end{equation}
Moreover, based on the second equation in (7.12), one easily obtains,
\begin{equation}
\begin{split}
\Delta_{g_S} u=\frac{1}{2}u^3|d\theta|_{g_S}^2.
\end{split}
\end{equation}

Now we analyze the operator $\beta_{g^{(4)}}\delta^*_{g^{(4)}}$ acting on a time-independent vector field $Y$, which is decomposed as in (7.8). To begin with, because the metric $g^{(4)}$ is vacuum, a standard Bochner-Weitzenbock formula gives,
\begin{equation*}
\begin{split}
2\beta_{g^{(4)}}\delta^*_{g^{(4)}}Y=\nabla^*\nabla Y-Ric_{g^{(4)}}(Y)=\nabla^*\nabla Y.
\end{split}
\end{equation*}
Based on the formula of the Laplace operator, we have,
\begin{equation}
\begin{split}
\nabla^*\nabla Y=\frac{1}{u^2}[\nabla_{\partial_t}\nabla_{\partial_t}Y-\nabla_{\nabla_{\partial_t}\partial_t}Y]-\Sigma_i[\nabla_{e_i}\nabla_{e_i}Y-\nabla_{\nabla_{e_i}e_i}Y],
\end{split}
\end{equation}
where $e_i~(i=1,2,3)$ are taken to be geodesic normal basis on $S$.
In the following, we compute the tensors on the right side of (7.15) term by term.
\\
\\1.We start with the first two terms in (7.15). Since $[Y,\partial_t]=0$, we have $\nabla_{\partial_t}Y=\nabla_Y\partial_t$. Thus, the first term in (7.15) gives
\begin{equation}
\begin{split}
\nabla_{\partial_t}\nabla_{\partial_t}Y&=\nabla_{\partial_t}\nabla_Y\partial_t=\nabla_{\nabla_Y\partial_t}\partial_t=\nabla_{\nabla_{Y^T}\partial_t}\partial_t-\frac{Y^{\perp}}{u}\nabla_{\nabla_{\partial_t}\partial_t}\partial_t\\
&=\nabla_{\nabla_{Y^T}\partial_t}\partial_t-\frac{Y^{\perp}}{u}\nabla_{u\nabla u}\partial_t.
\end{split}
\end{equation}
In the above, we use the decomposition (7.8) and the fact that $\nabla_{\partial_t}\partial_t=u\nabla u$. In the same way, the second term in (7.15) gives,
\begin{equation}
\begin{split}
&\nabla_{\nabla_{\partial_t}\partial_t}Y=\nabla_{u\nabla u}Y=\nabla_{u\nabla u}Y^T-\nabla_{u\nabla u}(\frac{Y^{\perp}}{u}\partial_t)\\
&=\nabla_{u\nabla u}Y^T-\langle u\nabla u,\nabla \frac{Y^{\perp}}{u}\rangle\cdot\partial_t-\frac{Y^{\perp}}{u}\nabla_{u\nabla u}\partial_t
\end{split}
\end{equation}
Subtract (7.17) from (7.16). We get
\begin{equation}
\begin{split}
\nabla_{\partial_t}\nabla_{\partial_t}Y-\nabla_{\nabla_{\partial_t}\partial_t}Y
&=\nabla_{\nabla_{Y^T}\partial_t}\partial_t-u\nabla_{\nabla u}Y^T+\langle u\nabla u,\nabla \frac{Y^{\perp}}{u}\rangle\cdot\partial_t
\end{split}
\end{equation}
Based on (7.6) and (7.8), 
\begin{equation}\nabla_v\partial_t=(\nabla_v\partial_t)^T-u^{-2}\langle\nabla_v\partial_t,\partial_t\rangle\cdot\partial_t\\
=-\frac{1}{2}u^2d\theta(v)+u^{-1}v(u)\cdot\partial_t\quad\forall v\in TS.\end{equation}
Thus the first term on the right side of (7.18) can be written as,
\begin{equation}
\begin{split}
&\nabla_{\nabla_{Y^T}\partial_t}\partial_t=\nabla_{-\frac{1}{2}u^2d\theta(Y^T)+u^{-1}Y^T(u)\cdot\partial_t}\partial_t\\
&=\nabla_{-\frac{1}{2}u^2d\theta(Y^T)}\partial_t+u^{-1}Y^T(u)\nabla_{\partial_t}\partial_t\\
&=\frac{1}{4}u^4d\theta(d\theta(Y^T))-\frac{1}{2}ud\theta(Y^T,\nabla u)\cdot\partial_t+Y^T(u)\nabla u.
\end{split}
\end{equation}
For two vector fields $v,w\in TS$, we have 
\begin{equation}
\nabla_vw=[\nabla_vw]^T+\langle\nabla_vw,\partial_t\rangle\cdot\frac{\partial_t}{-u^2}=(\nabla_{g_S})_vw+\frac{1}{2}d\theta(w,v)\cdot\partial_t.
\end{equation}
Applying the decomposition above to the second term on the right side of (7.18) yields,
\begin{equation}
\begin{split}
u\nabla_{\nabla u}Y^T=u(\nabla_{g_S})_{\nabla u}Y^T-\frac{1}{2}ud\theta(\nabla u,Y^T)\cdot\partial_t
\end{split}
\end{equation}
Plug (7.20),(7.22) into (7.18),
\begin{equation}
\begin{split}
&\nabla_{\partial_t}\nabla_{\partial_t}Y-\nabla_{\nabla_{\partial_t}\partial_t}Y\\
&=
\frac{1}{4}u^4d\theta(d\theta(Y^T))-\frac{1}{2}ud\theta(Y^T,\nabla u)\cdot\partial_t+Y^T(u)\nabla u\\
&\quad\quad-u(\nabla_{g_S})_{\nabla u}Y^T+\frac{1}{2}ud\theta(\nabla u,Y^T)\cdot\partial_t
+\langle u\nabla u,\nabla \frac{Y^{\perp}}{u}\rangle\cdot\partial_t\\
&=\frac{1}{4}u^4d\theta(d\theta(Y^T))+Y^T(u)\nabla u-u(\nabla_{g_S})_{\nabla u}Y^T
+\langle u\nabla u,\nabla \frac{Y^{\perp}}{u}\rangle\cdot\partial_t
\end{split}
\end{equation}
$~~~$\\
2.As for the third term in (7.15), we first use the decomposition (7.8) to get
\begin{equation}
\begin{split}
\nabla_{e_i}\nabla_{e_i}Y=\nabla_{e_i}\nabla_{e_i}Y^T-\nabla_{e_i}\nabla_{e_i}(\frac{Y^{\perp}}{u}\partial_t).
\end{split}
\end{equation}
Apply formula (7.21) to the first term on the right side above,
\begin{equation}
\begin{split}
&\nabla_{e_i}\nabla_{e_i}Y^T\\
&=\nabla_{e_i}[(\nabla_{g_S})_{e_i}Y^T+\frac{1}{2}d\theta(Y^T,e_i)\cdot\partial_t]\\
&=(\nabla_{g_S})_{e_i}(\nabla_{g_S})_{e_i}Y^T+\frac{1}{2}d\theta((\nabla_{g_S})_{e_i}Y^T,e_i)\cdot\partial_t+[\nabla_{e_i}\frac{1}{2}d\theta(Y^T,e_i)]\cdot\partial_t\\
&\quad\quad+\frac{1}{2}d\theta(Y^T,e_i)\cdot\nabla_{e_i}\partial_t\\
&=(\nabla_{g_S})_{e_i}(\nabla_{g_S})_{e_i}Y^T+\frac{1}{2}d\theta((\nabla_{g_S})_{e_i}Y^T,e_i)\cdot\partial_t+[\nabla_{e_i}\frac{1}{2}d\theta(Y^T,e_i)]\cdot\partial_t\\
&\quad\quad+\frac{1}{2}d\theta(Y^T,e_i)\cdot(-\frac{1}{2}u^2d\theta(e_i)+u^{-1}e_i(u)\cdot\partial_t)\\
&=(\nabla_{g_S})_{e_i}(\nabla_{g_S})_{e_i}Y^T-\frac{1}{4}u^2d\theta(Y^T,e_i)\cdot d\theta(e_i)\\
&\quad\quad+[\frac{1}{2}d\theta((\nabla_{g_S})_{e_i}Y^T,e_i)+\frac{1}{2}u^{-1}d\theta(Y^T, e_i)e_i(u)+\frac{1}{2}\nabla_{e_i}d\theta(Y^T,e_i)]\cdot\partial_t.
\end{split}
\end{equation}
Here in the third equality, we use the formula (7.19). As for the second term in (7.24), we first write it as,
\begin{equation}
\begin{split}
\nabla_{e_i}\nabla_{e_i}(\frac{Y^{\perp}}{u}\partial_t)&=\nabla_{e_i}[e_i(\frac{Y^{\perp}}{u})\partial_t+\frac{Y^{\perp}}{u}\nabla_{e_i}\partial_t]\\
&=e_i(e_i(\frac{Y^{\perp}}{u}))\partial_t+2e_i(\frac{Y^{\perp}}{u})\nabla_{e_i}\partial_t+\frac{Y^{\perp}}{u}\nabla_{e_i}\nabla_{e_i}\partial_t.
\end{split}
\end{equation}
Apply formula (7.19) to the second term above, 
\begin{equation}
\begin{split}
2e_i(\frac{Y^{\perp}}{u})\nabla_{e_i}\partial_t=2e_i(\frac{Y^{\perp}}{u})[-\frac{1}{2}u^2d\theta(e_i)+u^{-1}e_i(u)\cdot\partial_t].
\end{split}
\end{equation}
Apply formula (7.19) twice to the third term in (7.26) gives
\begin{equation}
\begin{split}
&\frac{Y^{\perp}}{u}\nabla_{e_i}\nabla_{e_i}\partial_t\\
&=\frac{Y^{\perp}}{u}\nabla_{e_i}[-\frac{1}{2}u^2d\theta(e_i)+u^{-1}e_i(u)\cdot\partial_t]\\
&=\frac{Y^{\perp}}{u}(\nabla_{g_S})_{e_i}[-\frac{1}{2}u^2d\theta(e_i)]-\frac{1}{4}uY^{\perp}d\theta(d\theta(e_i),e_i)\cdot\partial_t.\\
&\quad\quad+\frac{Y^{\perp}}{u}{e_i}(u^{-1}e_i(u))\cdot\partial_t+\frac{Y^{\perp}}{u^2}e_i(u)(-\frac{1}{2}u^2d\theta(e_i)+u^{-1}e_i(u)\cdot\partial_t)\\
&=\frac{Y^{\perp}}{u}(\nabla_{g_S})_{e_i}[-\frac{1}{2}u^2d\theta(e_i)]-\frac{1}{2}Y^{\perp}e_i(u)d\theta(e_i)\\
&\quad\quad+[\frac{1}{4}uY^{\perp}d\theta(e_i,d\theta(e_i))+\frac{Y^{\perp}}{u^2}{e_i}(e_i(u))]\cdot\partial_t.
\end{split}
\end{equation}
Summarizing equations (7.24-28) gives,
\begin{equation*}
\begin{split}
&\nabla_{e_i}\nabla_{e_i}Y\\
&=(\nabla_{g_S})_{e_i}(\nabla_{g_S})_{e_i}Y^T-\frac{1}{4}u^2d\theta(Y^T,e_i)\cdot d\theta(e_i)\\
&\quad\quad+[\frac{1}{2}d\theta((\nabla_{g_S})_{e_i}Y^T,e_i)+\frac{1}{2}u^{-1}d\theta(Y^T, e_i)e_i(u)+\frac{1}{2}\nabla_{e_i}d\theta(Y^T,e_i)]\cdot\partial_t\\
&\quad\quad-e_i(e_i(\frac{Y^{\perp}}{u}))\partial_t-2e_i(\frac{Y^{\perp}}{u})[-\frac{1}{2}u^2d\theta(e_i)+u^{-1}e_i(u)\cdot\partial_t]\\
&\quad\quad-\frac{Y^{\perp}}{u}(\nabla_{g_S})_{e_i}[-\frac{1}{2}u^2d\theta(e_i)]+\frac{1}{2}Y^{\perp}e_i(u)d\theta(e_i)\\
&\quad\quad-[\frac{1}{4}uY^{\perp}d\theta(e_i,d\theta(e_i))+\frac{Y^{\perp}}{u^2}{e_i}(e_i(u))]\cdot\partial_t\\
&=(\nabla_{g_S})_{e_i}(\nabla_{g_S})_{e_i}Y^T-\frac{1}{4}u^2d\theta(Y^T,e_i)\cdot d\theta(e_i)+e_i(\frac{Y^{\perp}}{u})u^2d\theta(e_i)\\
&\quad\quad-\frac{Y^{\perp}}{u}(\nabla_{g_S})_{e_i}[-\frac{1}{2}u^2d\theta(e_i)]+\frac{1}{2}Y^{\perp}e_i(u)d\theta(e_i)\\
&\quad\quad+[\frac{1}{2}d\theta((\nabla_{g_S})_{e_i}Y^T,e_i)+\frac{1}{2}u^{-1}d\theta(Y^T, e_i)e_i(u)+\frac{1}{2}\nabla_{e_i}d\theta(Y^T,e_i)]\cdot\partial_t\\
&\quad\quad-[e_i(e_i(\frac{Y^{\perp}}{u}))+2u^{-1}e_i(\frac{Y^{\perp}}{u})e_i(u)+\frac{1}{4}uY^{\perp}d\theta(e_i,d\theta(e_i))+\frac{Y^{\perp}}{u^2}{e_i}(e_i(u))]\cdot\partial_t
\end{split}
\end{equation*}
Take negative trace of the expression above,
\begin{equation}
\begin{split}
&-\Sigma_i\nabla_{e_i}\nabla_{e_i}Y\\
&=(\nabla_{g_S})^*\nabla_{g_S}Y^T+\frac{1}{4}u^2d\theta(d\theta(Y^T))-u^2d\theta(\nabla\frac{Y^{\perp}}{u})\\
&\quad\quad+\frac{Y^{\perp}}{2u}\delta_{g_S}[u^2d\theta]-\frac{1}{2}Y^{\perp}d\theta(\nabla u)\\
&\quad\quad+[\langle \frac{1}{2}d\theta,\nabla_{g_S}Y^T\rangle+\frac{1}{2}\delta_{g_S}(d\theta(Y^T))-\Delta_{g_S}(\frac{Y^{\perp}}{u})+\frac{1}{4}uY^{\perp}|d\theta|^2]\cdot\partial_t\\
&\quad\quad+[-\frac{1}{2}u^{-1}d\theta(Y^T,\nabla u)+2u^{-1}\langle\nabla\frac{Y^{\perp}}{u},\nabla u\rangle-\frac{Y^{\perp}}{u^2}\Delta_{g_S} u]\cdot\partial_t.
\end{split}
\end{equation}
Notice that  in the third line of (7.29), $\delta_{g_S}[u^2d\theta]=u^2\delta_{g_S}d\theta-2ud\theta(\nabla u)$. In the fourth line of (7.29), $\delta_{g_S}(d\theta(Y^T))=-\delta_{g_S}d\theta(Y^T)+\langle d\theta,\nabla_{g_S}Y^T\rangle$. Thus (7.29) can be rewritten as,
\begin{equation}
\begin{split}
&-\Sigma_i\nabla_{e_i}\nabla_{e_i}Y\\
&=(\nabla_{g_S})^*\nabla_{g_S}Y^T+\frac{1}{4}u^2d\theta(d\theta(Y^T))-u^2d\theta(\nabla\frac{Y^{\perp}}{u})\\
&\quad\quad+\frac{1}{2}Y^{\perp}u\delta_{g_S}[d\theta]-\frac{3}{2}Y^{\perp}d\theta(\nabla u)\\
&\quad\quad+[\langle d\theta,\nabla_{g_S}Y^T\rangle-\frac{1}{2}\delta_{g_S}d\theta(Y^T)-\Delta_{g_S}(\frac{Y^{\perp}}{u})+\frac{1}{4}uY^{\perp}|d\theta|^2]\cdot\partial_t\\
&\quad\quad+[-\frac{1}{2}u^{-1}d\theta(Y^T,\nabla u)+2u^{-1}\langle\nabla\frac{Y^{\perp}}{u},\nabla u\rangle-\frac{Y^{\perp}}{u^2}\Delta_{g_S} u]\cdot\partial_t.
\end{split}
\end{equation}
\\
4.The last term in (7.15) is zero because $\nabla_{e_i}e_i=0$ based on formula (7.21).
\\$~~$

Adding up the equations (7.23) and (7.30), we have
\begin{equation*}
\begin{cases}
[\nabla^*\nabla Y]^T=(\nabla_{g_S})^*\nabla_{g_S}Y^T+u^{-2}Y^T(u)\nabla u
-u^{-1}(\nabla_{g_S})_{\nabla u}Y^T\\
\quad\quad\quad\quad\quad\quad\quad\quad\quad\quad\quad\quad+\frac{1}{2}u^2d\theta(d\theta(Y^T))-u^2d\theta(\nabla\frac{Y^{\perp}}{u})\\
\quad\quad\quad\quad\quad\quad\quad\quad\quad\quad\quad\quad+\frac{1}{2}Y^{\perp}u\delta_{g_S}d\theta-\frac{3}{2}Y^{\perp}d\theta(\nabla u)\\
~~~\\
\langle \nabla^*\nabla Y,u^{-2}\partial_t\rangle=\Delta_{g_S}(\frac{Y^{\perp}}{u})-3u^{-1}\langle\nabla\frac{Y^{\perp}}{u},\nabla u\rangle-\langle d\theta,\nabla_{g_S}Y^T\rangle\\
\quad\quad\quad\quad\quad\quad\quad\quad\quad\quad-\frac{1}{4}uY^{\perp}|d\theta|^2+\frac{Y^{\perp}}{u^2}\Delta_{g_S} u\\
\quad\quad\quad\quad\quad\quad\quad\quad\quad\quad-\frac{1}{2}u^{-1}d\theta(\nabla u,Y^T)+\frac{1}{2}\delta_{g_S}d\theta(Y^T).
\end{cases}
\end{equation*}
According to equations (7.13) and (7.14), the equations above can be simplified as,
\begin{equation}
\begin{cases}
[\nabla^*\nabla Y]^T=(\nabla_{g_S})^*\nabla_{g_S}Y^T+u^{-2}Y^T(u)\nabla u
-u^{-1}(\nabla_{g_S})_{\nabla u}Y^T\\
\quad\quad\quad\quad\quad\quad\quad\quad\quad\quad\quad\quad\quad\quad\quad+\frac{1}{2}u^2d\theta(d\theta(Y^T))-u^2d\theta(\nabla\frac{Y^{\perp}}{u})\\
~~~\\
\langle \nabla^*\nabla Y,u^{-2}\partial_t\rangle=\Delta_{g_S}(\frac{Y^{\perp}}{u})-3u^{-1}\langle\nabla\frac{Y^{\perp}}{u},\nabla u\rangle+\frac{1}{4}uY^{\perp}|d\theta|^2\\
\quad\quad\quad\quad\quad\quad\quad\quad\quad\quad\quad\quad\quad\quad\quad-\langle d\theta,\nabla_{g_S}Y^T\rangle+u^{-1}d\theta(\nabla u,Y^T),
\end{cases}
\end{equation}
which is the formula (5.9).

We note that in the case where $\tilde g^{(4)}=\tilde g^{(4)}_0$,  the standard flat (Minkowski) metric on $\mathbb R\times (\mathbb R^3\setminus B)$. Because $\theta=0,~u=1$ for $\tilde g^{(4)}_0$, equations in (7.31) can be simplified as 
\begin{equation*}
\begin{cases}
[\nabla^*\nabla Y]^T=(\nabla_{g_0})^*\nabla_{g_0}Y^T\\
[\nabla^*\nabla Y]^{\perp}=\Delta_{g_0}Y^{\perp}.
\end{cases}
\end{equation*}
Here $g_0$ denotes the flat metric in $\mathbb R^3\setminus B$. Based on the decomposition above, it is easy to see that the solution to $\nabla^*\nabla Y=0$ with trivial Dirichlet boundary condition must be $Y=0$. Therefore, the operator $\beta_{\tilde g_0^{(4)}}\delta^*_{\tilde g_0^{(4)}}$ is invertible, i.e. the Assumption 3.1 holds for $\tilde g_0^{(4)}$.
\subsection{Perturbation of the operator $\beta_{\tilde g^{(4)}}\delta^*_{\tilde g^{(4)}}$}$~~$

Here we show that in the beginning of the proof of Proposition 5.3, if it is assumed that the system (5.6) admits a nontrivial solution  for all $\epsilon\in I$, then there exists a smooth curve $Y(\epsilon)$ solving it.

In the following discussion, we work with the weighted Sobolev spaces. Since a vector $Y$ solving BVP (5.6) must be $C^{\infty}$ smooth by elliptic regularity, the Banach space we choose does not affect the final conclusion. 
Let $M$ and $V^{(4)}$ be the same as in \S7.1. For fixed $p,\delta$, the weighted Sobolev spaces are defined as,
\begin{equation*}
\begin{split}
&L_{\delta}^p(M)=\{\text{functions u on }M:~||u||_{p,\delta}=(\int_{M}|u|^pr^{\delta p-n}dx)^{1/p}<\infty \},\\
&W_{\delta}^{k,p}(M)=\{\text{functions $u$ on }M:~\Sigma_{i=0}^k||D^iu||_{p,\delta+i}<\infty \},\\
&W^{k,p}(TV^{(4)})=\{\text{vector fields }Y\text{ in }V^{(4)}:~L_{\partial_t}Y=0, Y^{\gamma}\in W^{k,p}_{\delta}(M), \gamma=0,1,2,3\}.
\end{split}
\end{equation*}
Let $W$ be the space of vector fields that vanish on the boundary:
\begin{equation*}
\begin{split}
W=\{Y\in W^{2,2}_{\delta}(TV^{(4)}): Y=0 \text{ on }\partial M\}.
\end{split}
\end{equation*}

The operator $\beta_{\tilde g^{(4)}}\delta^*_{g_{\epsilon}^{(4)}}$ give rise to a family of map $T_{\epsilon}$ defined as,
\begin{equation*}
\begin{split}
&T_{\epsilon}:W\to L^2_{\delta}(TV^{4})\\
&T_{\epsilon}(Y)=r^{2}\beta_{\tilde g^{(4)}}\delta^*_{g_{\epsilon}^{(4)}}(Y)
\end{split}
\end{equation*}
It is obvious that $T_{\epsilon}$ is an analytic curve of linear operators parametrized by $\epsilon$. It has been proved in \S5 that $\beta_{\tilde g^{(4)}}\delta^*_{g_{\epsilon}^{(4)}}$ is formally self-adjoint, thus so is $T_{\epsilon}$. Moreover, by standard theory of elliptic operators on non-compact manifold (cf.[MV],[Le]), $T_{\epsilon}$ has compact resolvent.  According to [K] (Chapter 7, Theorem 3.9), for an analytic curve $T_{\epsilon}$ of self-adjoint operators that have compact resolvent, all (repeated) eigenvalues can be represented by analytic functions $u_n(\epsilon)$ and there is a sequence of analytic vector-valued functions $Y_n(\epsilon)$ representing the eigenvectors to $u_n(\epsilon)$. 

If, as assumed in the proof of Proposition 5.3, there is an interval $I$ such that for all $\epsilon\in I$ the system (5.6) admits a nonzero solution, then 0 is an eigenvalue of $T_{\epsilon}$ for all $\epsilon\in I$. Based on the analysis above, for each $\epsilon_0\in I$ there must be a function $u_n$ such that $u_n(\epsilon_0)=0$. However, there are only countably many eigenvalues $u_n$. Thus, among the eigenfunctions $u_n(\epsilon)$, there must be some $u_{n_0}$ such that $u_{n_0}(\epsilon)= 0$ for uncountably $\epsilon$. Since $u_{n_0}(\epsilon)$ is analytic in $\epsilon$, $u_{n_0}\equiv 0$ for $\epsilon\in I$. Correspondingly, $Y_{n_0}(\epsilon)$ is a smooth curve of 0-eigenvectors for $T_{\epsilon}$ ($\epsilon\in I$). This directly implies that there is a smooth curve $Y(\epsilon)$ solving the system (5.6).
\subsection{Bianchi operator in the Minkowski spacetime}$~~$

We give the proof of equation (6.6). Recall that the spacetime is $(V^{(4)}=\mathbb R\times(\mathbb R^3\setminus B),\tilde g_0^{(4)})$, where $\tilde g_0^{(4)}$ is the standard Minkowski metric. The metric is varied along the infinitesimal deformation $h^{(4)}$ such that 
\begin{equation}
\beta_{\tilde g_0^{(4)}}h^{(4)}=0.
\end{equation}
Under the standard coordinate $\{t,x^i\}$ of the flat spacetime $(V^{(4)},\tilde g_0^{(4)})$, the Killing vector field $\partial_t$ is of unit norm and it is perpendicular to the hypersurface $M=\{t=0\}$. On the hypersurface, ${\partial_{x^i}}~ (i=1,2,3)$ is a orthonormal basis of the tangent bundle. Let $\tilde{\nabla}$ denote the Levi-Civita connection of the flat metric. Then $\tilde{\nabla}_{\partial_t}\partial_t=0,~\tilde{\nabla}_{\partial_t}\partial_{x^i}=0$ and $\tilde{\nabla}_{\partial_{x^i}}\partial_{x^j}=0$. As in \S6, we use $g_0$ to denote the induced (flat) metric on $M$. 

While the infinitesimal variation of the spacetime metric is $h^{(4)}$, the shift vector is deformed by the vector field $Y\in TM$ such that $\langle Y,\partial_{x^i}\rangle_{g_0}=h^{(4)}(\partial_t,\partial_{x^i})$. Pairing (7.32) with $\partial_t$, we obtain,
\begin{equation*}
\begin{split}
0=\beta_{\tilde g_0^{(4)}}h^{(4)}(\partial_t)&=[\delta_{\tilde g_0^{(4)}}h^{(4)}+\frac{1}{2}dtrh^{(4)}](\partial_t)=\delta_{\tilde g_0^{(4)}}h^{(4)}(\partial_t)\\
&=\tilde{\nabla}_{\partial_t}h^{(4)}(\partial_t,\partial_t)-\Sigma_i\tilde{\nabla}_{\partial_{x^i}}h^{(4)}(\partial_{x^i},\partial_t)\\
&={\partial_t}\big(h^{(4)}(\partial_t,\partial_t)\big)-\Sigma_i{\partial_{x^i}}\big(h^{(4)}(\partial_{x^i},\partial_t)\big)\\
&=-\Sigma_i{\partial_{x^i}}\big(h^{(4)}(\partial_{x^i},\partial_t)\big)\\
&=\delta_{g_0} Y.
\end{split}
\end{equation*}
This gives the first equation in (6.6). In the calculation above, the second equality uses the fact that $h^{(4)}$ is time-independent. The third and last equality are based on that the metric $\tilde g_0^{(4)}$ and $g_0$ are flat. 

Under the deformation $h^{(4)}$, the induced metric on $M$ is deformed by $h$ which is the restriction of $h^{(4)}$ on $M$. The lapse function is deformed by $v$ so that $h^{(4)}(\partial_t,\partial_t)=-2v$. Pair (7.32) with $\partial_{x^i}$. Similar calculation as above gives,
\begin{equation*}
\begin{split}
0=\beta_{\tilde g_0^{(4)}}h^{(4)}(\partial_{x^i})&=[\delta_{\tilde g_0^{(4)}}+\frac{1}{2}dtrh^{(4)}](\partial_{x^i})\\
&=\delta_{\tilde g_0^{(4)}}h^{(4)}(\partial_{x^i})+\frac{1}{2}\partial_{x^i}(trh^{(4)})\\
&=\tilde{\nabla}_{\partial_t}h^{(4)}(\partial_t,\partial_{x^i})-\Sigma_k\tilde{\nabla}_{\partial_{x^k}}h^{(4)}(\partial_{x^k},\partial_{xi})+\frac{1}{2}\partial_{x^i}(trh^{(4)})\\
&={\partial_t}\big(h^{(4)}(\partial_t,\partial_{x^i})\big)-\Sigma_k{\partial_{x^k}}\big(h^{(4)}(\partial_{x^k},\partial_{x^i})\big)+\frac{1}{2}\partial_{x^i}(trh^{(4)})\\
&=-\Sigma_k{\partial_{x^k}}\big(h^{(4)}(\partial_{x^k},\partial_{x^i})\big)+\frac{1}{2}\partial_{x^i}(tr_{g_0}h^{(4)}-h^{(4)}(\partial_t,\partial_t))\\
&=\delta_{g_0} h(\partial_{x^i})+\frac{1}{2}d(tr_{g_0}h+2v)
\end{split}
\end{equation*}
This gives the second equation in (6.6). 
\begin{equation*}
\begin{split}
\end{split}
\end{equation*}
\bibliographystyle{amsplain}

\end{document}